%% file: main.tex
\documentclass[12pt,reqno]{amsart}

\usepackage{amsaddr}
\usepackage{fullpage}
\usepackage[dvipsnames]{xcolor}
\usepackage[hidelinks,bookmarks=false,linkcolor=black]{hyperref}
\usepackage{graphicx}
\usepackage{natbib}

\usepackage{amssymb, amsmath, amsfonts, amsthm}
\usepackage{bm}
\usepackage{mathtools}
\mathtoolsset{showonlyrefs=true}

\newtheorem{theorem}{Theorem}
\newtheorem{corollary}{Corollary}
\newtheorem{lemma}{Lemma}
\newtheorem{proposition}{Proposition}
\newtheorem{assumption}{Assumption}
\theoremstyle{definition}

\newtheorem{remark}{Remark}
\newtheorem{example}{Example}

\usepackage{pgfplots}
\usepackage{subcaption}
\pgfplotsset{compat=1.18}

\title{Priors for second-order unbiased Bayes estimators}

\author{Mana Sakai$^{\ast\dagger}$
\quad
Takeru Matsuda$^{\ast\ddagger}$
\quad
Tatsuya Kubokawa$^{\ast}$
}

\address{
$^{\ast}$The University of Tokyo
\\
$^{\dagger}$RIKEN Center for Advanced Intelligence Project
\\
$^{\ddagger}$RIKEN Center for Brain Science
\footnote{\textit{Email}: mana.sakai.77@gmail.com, matsuda@mist.i.u-tokyo.ac.jp, tatsuya@e.u-tokyo.ac.jp}
}

\begin{document}
\begin{abstract}
    Asymptotically unbiased priors, introduced by \cite{hartigan1965asymptotically}, are designed to achieve second-order unbiasedness of Bayes estimators. This paper extends Hartigan's framework to non-i.i.d. models by deriving a system of partial differential equations that characterizes asymptotically unbiased priors. Furthermore, we establish a necessary and sufficient condition for the existence of such priors and propose a simple procedure for constructing them. The proposed method is applied to the linear regression model and the nested error regression model (also known as the random effects model). Simulation studies evaluate the frequentist properties of the Bayes estimator under the asymptotically unbiased prior for the nested error regression model, highlighting its effectiveness in small-sample settings.
\end{abstract}

\maketitle

\section{Introduction}

\subsection{Background}

In Bayesian inference, the choice of prior distribution is a crucial consideration. When researchers have prior knowledge about the parameter of interest, they can incorporate this information into the prior. In the absence of prior knowledge, however, it is common to use so-called objective priors. Well-known examples of objective priors include Jeffreys' prior \citep{jeffreys1991theory}, probability matching priors \citep{welch1963formulae,datta1995priors} and reference priors \citep{bernardo1979reference,berger1992development}, among others.

In addition to these objective priors, \cite{hartigan1965asymptotically} introduced the concept of asymptotically unbiased priors, which are designed to achieve second-order unbiasedness of Bayes estimators. While Hartigan's original work considered general loss functions, we focus on the squared error loss, for which the posterior mean is the Bayes estimator. We define bias as its difference from the true parameter value.

\cite{hartigan1965asymptotically} showed that, in general, the posterior mean has a bias of $O(n^{-1})$, where $n$ is the sample size. He defined the asymptotically unbiased prior as one that reduces this bias to $o(n^{-1})$, achieving second-order unbiasedness. \cite{hartigan1965asymptotically}  considered i.i.d. models and derived conditions that asymptotically unbiased priors should satisfy. For one-dimensional parameter models, the asymptotically unbiased priors are proportional to the Fisher information of the models. For multi-parameter models, such priors are characterized as solutions to a system of partial differential equations.

The utility of asymptotically unbiased priors lies in their ability to reduce bias of the Bayes estimator, which is particularly important when the sample size is small and bias tends to be large. In this sense, asymptotically unbiased priors are similar to the bias reduction approach for maximum likelihood estimators introduced by \cite{firth1993bias}. The connection is more than just a conceptual similarity. As shown in Remark \ref{rem:Firth}, for i.i.d. models, the condition that the log-prior must satisfy to be asymptotically unbiased is the sum of two other key terms from the literature: the term defining the moment matching prior \citep{ghosh2011moment}, which is a prior designed to match the posterior mean with the maximum likelihood estimator up to $o_{p}(n^{-1})$, and the score modification term from Firth's method. Therefore, the asymptotically unbiased prior can be interpreted as a bias-reduced version of the moment matching prior, where the bias reduction is achieved using Firth's method.

\subsection{Motivation and contributions}

The results of \cite{hartigan1965asymptotically}  are significant, but the restriction to i.i.d. models often proves too limiting in practice. For instance, in Bayesian regression problems, it is customary to treat the response variables as random while assuming the explanatory variables are fixed. In such cases, the data are not i.i.d. and the results of \cite{hartigan1965asymptotically} do not directly apply.

In this paper, we extend the findings of \cite{hartigan1965asymptotically}  to non-i.i.d. models, deriving a system of partial differential equations that characterizes asymptotically unbiased priors. The generalization broadens the applicability of asymptotically unbiased priors to a wider range of models.

Moreover, we establish a necessary and sufficient condition for the existence of such priors and present a simple procedure for constructing them. These results are particularly valuable, as it provides a unified framework for constructing asymptotically unbiased priors across various models. This condition and the construction method are also applicable to other classes of priors $\pi(\theta)$ for a $p$-dimensional parameter $\theta$, which are defined as solutions to a system of partial differential equations of the form $\partial\log\pi(\theta)/\partial\theta_{r}=A_{r}(\theta)\ (r=1,\ldots,p)$, such as the moment matching priors proposed by \cite{ghosh2011moment}.

To demonstrate the usefulness of our construction method, we apply it to the linear regression model and the nested error regression model. In the linear regression model, we illustrate how the asymptotically unbiased prior depends on the chosen parameterization and show that our method successfully derives priors that lead to exactly unbiased estimators. For the nested error regression model, although an analytical form of the posterior mean with the asymptotically unbiased prior is not available, we evaluate its performance through simulation studies. The prior we construct for the nested error regression model, to the best of our knowledge, has not been explored in the existing literature. Our simulation studies highlight its effectiveness in small samples (i.e. small number of areas/groups), particularly in terms of bias reduction.

Our contributions are threefold. First, we extend the work of \cite{hartigan1965asymptotically} by deriving the system of partial differential equations for asymptotically unbiased priors in a general non-i.i.d. setting. Second, we establish a necessary and sufficient condition for their existence and provide a simple construction procedure. Third, we propose a novel asymptotically unbiased prior for the nested error regression/random effects model. Simulation studies demonstrate that the proposed prior reduces bias effectively in small samples.

\section{Main results}\label{sec:MainResults}

\subsection{Second-order bias of Bayes estimators}\label{subsec:SecondOrderBias}

In the following two subsections, we present the main theoretical results of this paper. This subsection begins with the derivation of the second-order bias of Bayes estimators under simple regularity conditions (Corollary~\ref{cor:Bias.DepOnN}). Building on this result, with additional assumptions, we further simplify the evaluation of bias (Corollary~\ref{cor:Bias.IndepOfN}). All the proofs of the following two subsections are provided in Appendix~\ref{sec:Proofs}.

Suppose $X_{1},\ldots,X_{n}$ has a joint density function $f_{n}(x_{1},\ldots,x_{n}\mid\theta)$. Here, $\theta$ is an interior point of the parameter space $\Theta$, which is a rectangular subspace of $\mathbb{R}^{p}$. The log-likelihood function is denoted by $\ell_{n}(\theta)=\log f_{n}(x_{1},\ldots,x_{n}\mid\theta)$. To find the maximum likelihood estimator $\hat{\theta}^{ML}$, we need to solve the equation $\partial\ell_{n}(\theta)/\partial\theta\mid_{\theta=\hat{\theta}^{ML}}=0$. Let $\hat{\theta}^{B}$ denote the posterior mean, which we call the (generalized) Bayes estimator, corresponding to a prior density $\pi$. It can be calculated as
\[
  \hat{\theta}^{B}
  =\frac{\int_{\Theta}\theta\pi(\theta)f_{n}(x_{1},\ldots,x_{n}\mid\theta)d\theta}{\int_{\Theta}\pi(\theta)f_{n}(x_{1},\ldots,x_{n}\mid\theta)d\theta}
  =\frac{\int_{\Theta}\theta\pi(\theta)\exp\{\ell_{n}(\theta)\}d\theta}{\int_{\Theta}\pi(\theta)\exp\{\ell_{n}(\theta)\}d\theta}
  .
\]
The negative second-order derivative of the log-likelihood is denoted by
\[
  H_{n}(\theta)
  =-n^{-1}\frac{\partial^{2}\ell_{n}(\theta)}{\partial\theta\partial\theta^{\mathrm{\scriptscriptstyle{T}}}}
  .
\]

We assume the following regularity conditions.
\begin{assumption}\label{asm:Limits1}
  Assume that the following conditions are satisfied: (i) $\hat{\theta}^{ML}-\theta=O_{p}(n^{-1/2})$; (ii) $\ell_{n}(\theta)=O_{p}(n)$; (iii) $\ell_{n}(\theta)$ is three times continuously differentiable; (iv) the prior density $\pi(\theta)$ is differentiable; (v) $H_{n}(\theta)$ is invertible and $H_{n}(\hat{\theta}^{ML})$ is positive definite.
\end{assumption}

Note that conditions (i) and (ii) imply $\partial\ell_{n}(\theta)/\partial\theta=O_{p}(n^{1/2})$. For notational simplicity, we write
\begin{align*}
  &I_{n,rs}(\theta)
  =n^{-1}\frac{\partial\ell_{n}(\theta)}{\partial\theta_{r}}\frac{\partial\ell_{n}(\theta)}{\partial\theta_{s}}
  ,\quad
  J_{n,rs,t}(\theta)
  =n^{-1}\frac{\partial^{2}\ell_{n}(\theta)}{\partial\theta_{r}\partial\theta_{s}}\frac{\partial\ell_{n}(\theta)}{\partial\theta_{t}}
  ,\\
  &K_{n,rst}(\theta)
  =n^{-1}\frac{\partial^{3}\ell_{n}(\theta)}{\partial\theta_{r}\partial\theta_{s}\partial\theta_{t}}
  .
\end{align*}
Additionally, we denote the $(r,s)$th element of any matrix $M$ as $M_{rs}$ and the $(r,s)$th element of its inverse (if it exists) as $M^{rs}$.

\begin{theorem}\label{thm:BayesEst.True.Diff}
  Suppose the model satisfies Assumption~\ref{asm:Limits1}. Then, the Bayes estimator can be decomposed as
  \begin{equation}\label{eq:BayesEst.True.Diff1}
    \hat{\theta}^{B}
    =\theta+n^{-1}H_{n}^{-1}(\theta)\left[\frac{\partial}{\partial\theta}\{\log\pi(\theta)+2\ell_{n}(\theta)\}+\frac{1}{2}\sum_{r=1}^{p}\sum_{s=1}^{p}H_{n}^{rs}(\theta)A_{n,rs}(\theta)\right]+o_{p}(n^{-1})
    ,
  \end{equation}
  where each element of $p$-dimensional vector
  $A_{n,rs}(\theta)\ (r,s=1,\ldots,p)$ is
  \[
    A_{n,rsv}(\theta)
    =K_{n,vrs}(\theta)+2J_{n,vr,s}(\theta)+\sum_{t=1}^{p}\sum_{u=1}^{p}H_{n}^{tu}(\theta)I_{n,su}(\theta)K_{n,rtv}(\theta)
    \quad(v=1,\ldots,p)
    .
  \]
\end{theorem}

Assuming that the expectation of the remainder term $o_{p}(n^{-1})$ in \eqref{eq:BayesEst.True.Diff1} is $o(n^{-1})$, we can evaluate the bias of the Bayes estimator as follows.

\begin{corollary}\label{cor:Bias.DepOnN}
  Suppose the model satisfies Assumption~\ref{asm:Limits1}. Then, the bias of the Bayes estimator is
  \begin{equation}\label{eq:BayesEstBias}
    E(\hat{\theta}^{B})-\theta
    =n^{-1}E\left(H_{n}^{-1}(\theta)\left[\frac{\partial}{\partial\theta}\left\{\log\pi(\theta)+2\ell_{n}(\theta)\right\}+\frac{1}{2}\sum_{r=1}^{p}\sum_{s=1}^{p}H_{n}^{rs}(\theta)A_{n,rs}(\theta)\right]\right)+o(n^{-1})
    .
  \end{equation}
  If, in particular, $H_{n}(\theta)$ is non-stochastic (and so is $K_{n,rst}$), the bias of the Bayes estimator can be expressed as
  \[
    E(\hat{\theta}^{B})-\theta
    =n^{-1}H_{n}^{-1}(\theta)\left[\frac{\partial\log\pi(\theta)}{\partial\theta}+\frac{1}{2}\sum_{r=1}^{p}\sum_{s=1}^{p}H_{n}^{rs}(\theta)E\{A_{n,rs}(\theta)\}\right]+o(n^{-1})
    .
  \]
  Thus, the Bayes estimator is second-order unbiased if the prior satisfies
  \begin{equation}\label{eq:PriorCond.DepOnN}
    \frac{\partial\log\pi(\theta)}{\partial\theta}
    =-\frac{1}{2}\sum_{r=1}^{p}\sum_{s=1}^{p}H_{n}^{rs}(\theta)E\{A_{n,rs}(\theta)\}
    \equiv\phi(\theta)
    .
  \end{equation}
  In this case, $E\{A_{n,rs}(\theta)\}$ has the form of
  \[
    E\{A_{n,rsv}(\theta)\}
    =K_{n,vrs}(\theta)+2E\{J_{n,vr,s}(\theta)\}
    +\sum_{t=1}^{p}\sum_{u=1}^{p}H_{n}^{tu}(\theta)E\{I_{n,su}(\theta)\}K_{n,rtv}(\theta)
    .
  \]
\end{corollary}

It should be noted that the prior satisfying \eqref{eq:PriorCond.DepOnN} depends on the sample size $n$.

It seems relatively easier to find a prior that ensures second-order unbiasedness of the Bayes estimator if $H_{n}(\theta)$ is non-stochastic. However, in general, $H_{n}(\theta)$ is a random matrix. In such cases, we need alternative strategies to find a prior that leads to a second-order unbiased Bayes estimator. One possible approach is to consider the asymptotic properties of elements such as $H_{n}(\theta)$ and $K_{n, rst}(\theta)$. With some additional assumptions, we can simplify the evaluation of the bias of the Bayes estimator given in \eqref{eq:BayesEstBias}.

\begin{assumption}\label{asm:Limits2}
  Assume that the following conditions are satisfied: (i) there exists a non-singular $p\times p$ constant matrix $H(\theta)$ that satisfies $H_{n}(\theta)=H(\theta)+O_{p}(n^{-1/2})$, and consequently, $H_{n}^{-1}(\theta)=H^{-1}(\theta)+O_{p}(n^{-1/2})$; (ii) for each $r,s,t=1,\ldots,p$, there exists a constant $K_{rst}(\theta)$ that satisfies $K_{n,rst}(\theta)=K_{rst}(\theta)+o_{p}(1)$; (iii) there exists a $p\times p$ matrix $I(\theta)$ that satisfies $E\{I_{n}(\theta)\}=I(\theta)+o(1)$; (iv) for each $r,s,t=1,\ldots,p$, there exists $J_{rs,t}(\theta)$ that satisfies $E\{J_{n,rs,t}(\theta)\}=J_{rs,t}(\theta)+o(1)$.
\end{assumption}

With Assumption~\ref{asm:Limits2} in place, we can now proceed to simplify the evaluation of the bias of the Bayes estimator, taking advantage of the asymptotic behavior of the matrices involved. Specifically, these conditions allow us to express the bias in a form that does not depend on the sample size $n$.

\begin{corollary}\label{cor:Bias.IndepOfN}
  Suppose the model satisfies Assumptions~\ref{asm:Limits1} and \ref{asm:Limits2}. Then, the bias of the Bayes estimator is
  \[
    E(\hat{\theta}^{B})-\theta
    =n^{-1}H^{-1}(\theta)\left\{\frac{\partial\log\pi(\theta)}{\partial\theta}+\frac{1}{2}\sum_{r=1}^{p}\sum_{s=1}^{p}H^{rs}(\theta)A_{rs}(\theta)\right\}+o(n^{-1})
    ,
  \]
  where each element of $p$-dimensional vector $A_{rs}(\theta)\ (r,s=1,\ldots,p)$ is defined as
  \[
    A_{rsv}(\theta)
    =K_{vrs}(\theta)+2J_{vr,s}(\theta)+\sum_{t=1}^{p}\sum_{u=1}^{p}H^{tu}(\theta)I_{su}(\theta)K_{rtv}(\theta)
    \quad(v=1,\ldots,p)
    .
  \]
  This implies that the Bayes estimator is second-order unbiased if the prior $\pi(\theta)$ satisfies
  \begin{equation}\label{eq:PriorCond.IndepOfN}
    \frac{\partial\log\pi(\theta)}{\partial\theta}
    =-\frac{1}{2}\sum_{r=1}^{p}\sum_{s=1}^{p}H^{rs}(\theta)A_{rs}(\theta)
    \equiv\phi(\theta)
    .
  \end{equation}
\end{corollary}

\subsection{Existence of asymptotically unbiased priors}\label{subsec:ExistenceOfPrior}

An asymptotically unbiased prior, if it exists, can be obtained by solving \eqref{eq:PriorCond.DepOnN} (assuming $H_{n}(\theta)$ is non-stochastic) or \eqref{eq:PriorCond.IndepOfN} (assuming the model satisfies Assumption~\ref{asm:Limits2}). However, there is no guarantee that such priors exist for all models. To address this issue, we examine the necessary and sufficient condition for these equations to have a solution. The following theorem applies a classical result from the theory of partial differential equations.

\begin{theorem}\label{thm:ExistenceOfPrior}
  Let $\phi:\Theta\to\mathbb{R}^{p}$ be a differentiable vector-valued function, and assume that the order of integration and differentiation in $\phi$ can be interchanged, i.e., $\int(\partial\phi(\theta)/\partial \theta)d\theta=\partial(\int\phi(\theta)d\theta)/\partial\theta$ holds. Then, a twice-differentiable prior density $\pi(\theta)$ satisfying
  \begin{equation}\label{eq:PriorCond}
    \frac{\partial\log\pi(\theta)}{\partial\theta}=\phi(\theta)
  \end{equation}
  exists if and only if the following holds:
  \begin{equation}\label{eq:ExistenceOfPrior}
    \frac{\partial\phi_{t}(\theta)}{\partial\theta_{u}}
    =\frac{\partial\phi_{u}(\theta)}{\partial\theta_{t}}
    \quad(t,u=1,\ldots,p)
    .
  \end{equation}
  We refer to \eqref{eq:ExistenceOfPrior} as the integrability condition.
\end{theorem}

When considering the existence of an asymptotically unbiased prior, the function $\phi$ in Theorem~\ref{thm:ExistenceOfPrior} corresponds to the right-hand side of either \eqref{eq:PriorCond.DepOnN} or \eqref{eq:PriorCond.IndepOfN}.

The following corollary offers a practical method for constructing an asymptotically unbiased prior. This builds on the sufficiency part of the proof of Theorem~\ref{thm:ExistenceOfPrior}.

\begin{corollary}\label{cor:ConstructionOfPrior}
  Suppose the model satisfies the assumptions of Theorem~\ref{thm:ExistenceOfPrior}. If the integrability condition \eqref{eq:ExistenceOfPrior} holds, a prior constructed using the following procedure satisfies \eqref{eq:PriorCond}.
  \begin{enumerate}
    \item Fix a constant vector $(c_{1},\ldots,c_{p})\in\Theta$ arbitrarily.
    \item Define $\psi_{t}(\theta_{t},\ldots,\theta_{p})=\phi_{t}(c_{1},\ldots,c_{t-1},\theta_{t},\ldots,\theta_{p})$ and compute
    \[
      \tilde{\pi}(\theta)
      =\exp \left\{\sum_{t=1}^{p}\int_{c_{t}}^{\theta_{t}}\psi_{t}(z,\theta_{t+1},\ldots,\theta_{p})dz\right\}
      .
    \]
    \item Take the prior as $\pi(\theta)\propto \tilde{\pi}(\theta)$.
  \end{enumerate}
\end{corollary}

The prior constructed using the procedure in Corollary~\ref{cor:ConstructionOfPrior} may often be improper, meaning that its integral over the parameter space diverges to infinity. This does not invalidate the posterior mean, as long as the posterior distribution is proper (i.e. integrable).

\begin{remark}[Comparison with probability-matching priors]
  The asymptotically unbiased prior aims at removing the $O(n^{-1})$ bias of the posterior mean, whereas the probability matching prior is designed so that posterior quantiles attain the frequentist coverage probability up to a certain order. The resulting priors are therefore generally different. For a one-dimensional parameter i.i.d. model, for instance, the first-order probability matching prior coincides with the Jeffreys prior $\pi(\theta)\propto I(\theta)^{1/2}$ \citep{welch1963formulae}. In contrast, the asymptotically unbiased prior is proportional to the Fisher information itself, $\pi(\theta)\propto I(\theta)$ (Appendix~\ref{subsec:OneParameterModels}). This distinction implies a fundamental trade-off. In general, the credible intervals under the asymptotically unbiased priors are not (first or second-order) probability-matching. Conversely, posterior means under probability matching priors typically retain an $O(n^{-1})$ bias. Consequently, the choice between these classes of priors should be driven by the inferential goal: bias reduction for point estimation versus coverage accuracy for interval estimation.
\end{remark}

\section{Examples}\label{sec:Examples}

In this section, we present examples of models where an asymptotically unbiased prior can be constructed using the procedure described above. We begin with the linear regression model, considering two different parametrizations. In both cases, the Bayes estimator under the asymptotically unbiased prior is exactly unbiased. Then, we consider the nested error regression/random effects model, which is widely used in many fields, including small area estimation. For more details and applications of the nested error regression model in small area estimation, please refer to, for example, \cite{sugasawa2023mixed} and \cite{rao2015small}.

\begin{example}[Linear regression model]\label{ex:LinearReg}
  Consider the model $y\sim N(X\beta,\sigma^{2}I_{n})$, where $y$ is an $n$-dimensional random vector of response variables, $X$ is an $n\times p$ non-random matrix of explanatory variables and $\beta\in\mathbb{R}^{p}$ is the coefficient vector. First, we consider the estimation problem of the parameter $\theta=(\beta,\sigma^{2})\in\mathbb{R}^{p}\times(0,\infty)$. As derived in Appendix~\ref{subsec:RegModel}, the asymptotically unbiased prior is $\pi(\beta,\sigma^{2})\propto\sigma^{-4}$. The posterior mean under this prior is $(\hat{\beta}^{B},\hat{\sigma}^{2,B})=(\hat{\beta}^{OLS},y^{\mathrm{\scriptscriptstyle{T}}}(y-X\hat{\beta}^{OLS})/(n-p))$ for $n>p$, where $\hat{\beta}^{OLS}=(X^{\mathrm{\scriptscriptstyle{T}}}X)^{-1}X^{\mathrm{\scriptscriptstyle{T}}}y$ is the ordinary least squares estimator. It is an exactly unbiased estimator of $(\beta,\sigma^{2})$.

  This unbiasedness, however, is specific to the parameterization. For example, if the primary inferential goal is to estimate an $\alpha$-quantile of a new observation, such as $q_{\alpha}=x_{new}^{\mathrm{\scriptscriptstyle{T}}}\beta+z_{\alpha}\sigma$ with $z_{\alpha}$ being the $\alpha$-quantile of the standard normal distribution, the prior $\pi(\beta,\sigma^{2})\propto\sigma^{-4}$ leads to a second-order biased estimator of $q_{\alpha}$. To obtain an unbiased estimator for the quantile, one can instead consider the parameterization $\theta'=(\beta,\sigma)\in\mathbb{R}^{p}\times(0,\infty)$. For this choice, the asymptotically unbiased prior is $\pi(\beta,\sigma)\propto\sigma^{-2}$. This prior leads to a Bayes estimator that is exactly unbiased for $(\beta,\sigma)$ and, by linearity, the quantile $q_{\alpha}$. The full derivation and the explicit form of the estimator are given in Appendix~\ref{subsec:RegModel}.

  This example illustrates a key principle: the choice of parameterization and the resulting asymptotically unbiased prior should be guided by the specific inferential goal of the analysis. While the prior $\pi(\beta,\sigma^{2})\propto\sigma^{-4}$ leads to an unbiased estimator of $(\beta,\sigma^{2})$, it yields a biased estimator of the quantile $q_{\alpha}$. In contrast, the prior $\pi(\beta,\sigma)\propto\sigma^{-2}$ ensures the Bayes estimator of $q_{\alpha}$ is unbiased.
\end{example}

\begin{example}[Balanced nested error regression model]\label{ex:NER.CommonN}
  We consider the balanced nested error regression model $y_{ij}=x_{ij}^{\mathrm{\scriptscriptstyle{T}}}\beta+v_{i}+\epsilon_{ij}\ (i=1,\ldots,m;\ j=1,\ldots,n)$, where $y_{ij}$ is the response variable for each unit $j$ in area $i$, $x_{ij}$ is a $p$-dimensional non-random covariate vector, $\beta\in\mathbb{R}^{p}$ is the coefficient vector. The random efecrts for area $i$, $v_{i}$, and the error term, $\epsilon_{ij}$, are assumed to be independent and follow normal distributions $v_{i}\sim N(0,\tau^{2})$ and $\epsilon_{ij}\sim N(0,\sigma^{2})$, respectively. We consider the asymptotic setting where $m\to\infty$, with $n$ fixed. The parameter of interest is $\theta=(\theta_{1},\ldots,\theta_{p},\theta_{p+1},\theta_{p+2})=(\beta,\tau^{2},\sigma^{2})\in\Theta=\mathbb{R}^{p}\times(0,\infty)^{2}$. By Appendix~\ref{subsec:RegModel}, we obtain an asymptotically unbiased prior as
  \begin{equation}\label{eq:NER.Prior.CommonN}
    \pi(\theta)\propto\sigma^{-4}(\sigma^{2}+n\tau^{2})^{-2}
    .
  \end{equation}
  The posterior propriety of the prior in \eqref{eq:NER.Prior.CommonN} and a description of a Gibbs sampler for posterior sampling is discussed in Appendix~\ref{sec:DiscussNER}.
\end{example}

\section{Simulation studies on the nested error regression model}\label{sec:Simulation}

Most of the examples presented in Section~\ref{sec:Examples} and Appendix~\ref{sec:Examlpes.Appendix} yield analytical posterior distributions and exactly unbiased Bayes estimators. However, for the nested error regression model, deriving an analytical posterior distribution is not feasible, and consequently, the properties of the Bayes estimator remain unknown. We therefore conduct simulation studies to evaluate the frequentist properties of our proposed prior in this model, particularly focusing on the small-sample (i.e. small number of areas/groups) performance common in applications like small area estimation. We compare the performance of the proposed asymptotically unbiased prior with Jeffreys' prior for the variance components and the hierarchical prior proposed by \cite{datta1991bayesian}. The full details of the simulation design including the precise specification of the priors, and MCMC settings are provided in Appendix~\ref{app:Simulation}. All simulation code is available at \url{https://github.com/manasakai/second-order-unbiased-NER}.

\input{plots/AbsBias}

Figure~\ref{fig:AbsBias} shows the absolute bias of the posterior mean of the variance components $\tau^{2}$ and $\sigma^{2}$. The choice of prior strongly influences the bias, especially for small sample sizes ($m=10,\ 32$). The proposed prior exhibits smaller bias for $\tau^{2}$ under both parameter settings. For $\sigma^{2}$, Jeffreys' prior can be advantageous in certain settings. These findings suggest that while the asymptotically unbiased prior provides robust performance for variance components in general, alternative priors, such as Jeffreys' prior, may excel in specific parameter settings. We note that the bias of $\beta$ remains relatively small and stable across priors and sample sizes, as shown in Appendix~\ref{app:Simulation}. It also provides the evaluation of the mean squared error and the coverage probability of the 95\% credible interval.

\section*{Acknowledgements}

We would like to thank Kaoru Irie, Shonosuke Sugasawa, and Shintaro Hashimoto for their valuable comments on the content of this article. We also thank Tomoya Wakayama for his assistance in improving the simulation code. Takeru Matsuda was supported by JSPS KAKENHI Grant Numbers 21H05205, 22K17865, 24K02951 and JST Moonshot Grant Number JPMJMS2024. Tatsuya Kubokawa was supported by JSPS KAKENHI Grant Number 22K11928.

\vfill
\pagebreak

\appendix

\numberwithin{equation}{section}
\numberwithin{theorem}{section}
\numberwithin{corollary}{section}
\numberwithin{lemma}{section}
\numberwithin{proposition}{section}
\numberwithin{assumption}{section}
\numberwithin{remark}{section}
\numberwithin{example}{section}

\section{Proofs of the main theoretical results}\label{sec:Proofs}

\subsection{Proof of Theorem~\texorpdfstring{\ref{thm:BayesEst.True.Diff}}{1}}\label{subsubsec:Prf_thm:BayesEst.True.Diff}

The proof strategy is to evaluate $\hat{\theta}^{B}-\hat{\theta}^{ML}$ and $\hat{\theta}^{ML}-\theta$ separately, and combine them to obtain the desired decomposition.

We sometimes include dots instead of indices $r,s,t=1,\ldots,p$ to denote a vector or a matrix composed of $K_{n,rst}$. For example, we write
\[
  K_{n,r\cdot t}(\theta)
  =\begin{bmatrix}
    K_{n,r1t}(\theta)\\
    \vdots\\
    K_{n,rpt}(\theta)
  \end{bmatrix}
  ,\quad
  K_{n,r\cdot\cdot}(\theta)
  =\begin{bmatrix}
    K_{n,r11}(\theta)&\cdots&K_{n,r1p}(\theta)\\
    \vdots&\ddots&\vdots\\
    K_{n,rp1}(\theta)&\cdots&K_{n,rpp}(\theta)
  \end{bmatrix}
  .
\]

We first evaluate $\hat{\theta}^{B}-\hat{\theta}^{ML}$ using Laplace's method.

\begin{lemma}\label{lem:BayesEst.ML.Diff}
  Suppose the model satisfies Assumption~\ref{asm:Limits1}. Then, we can decompose the difference between the Bayes estimator and the maximum likelihood estimator as
  \begin{equation}\label{eq:BayesEst.ML.Diff}
    \hat{\theta}^{B}-\hat{\theta}^{ML}
    =n^{-1}H_{n}^{-1}(\theta)\left\{\frac{\partial\log\pi(\theta)}{\partial\theta}+\frac{1}{2}\sum_{r=1}^{p}\sum_{s=1}^{p}H_{n}^{rs}(\theta)K_{n,\cdot rs}(\theta)\right\}+o_{p}(n^{-1})
    .
  \end{equation}
\end{lemma}
\begin{proof}
  Noting that $\pi(\theta)$ is null if $\theta\notin\Theta$, we can write the Bayes estimator as
  \[
    \hat{\theta}^{B}
    =\frac{\int_{\mathbb{R}^{p}}\theta\pi(\theta)\exp\{\ell_{n}(\theta)\}d\theta}{\int_{\mathbb{R}^{p}}\pi(\theta)\exp\{\ell_{n}(\theta)\}d\theta}
    .
  \]
  Here, we can expand $\pi(\theta)$ around $\hat{\theta}^{ML}$ as
  \[
    \pi(\theta)
    =\pi(\hat{\theta}^{ML})+\dot{\pi}(\hat{\theta}^{ML})^{\mathrm{\scriptscriptstyle{T}}}(\theta-\hat{\theta}^{ML})+O_{p}(n^{-1})
    ,
  \]
  where $\dot{\pi}(\hat{\theta}^{ML})$ is the first order derivative of $\pi(\theta)$ at $\hat{\theta}^{ML}$. Similarly, we can expand $\ell_{n}(\theta)$ around $\hat{\theta}^{ML}$ as
  \[
    \ell_{n}(\theta)
    =\ell_{n}(\hat{\theta}^{ML})-\frac{n}{2}(\theta-\hat{\theta}^{ML})^{\mathrm{\scriptscriptstyle{T}}}H_{n}(\hat{\theta}^{ML})(\theta-\hat{\theta}^{ML})+\frac{1}{6}\left\{\sum_{r=1}^{p}(\theta_{r}-\hat{\theta}_{r}^{ML})\frac{\partial}{\partial\theta_{r}}\right\}^{3}\ell_{n}(\hat{\theta}^{ML})+O_{p}(n^{-1})
    ,
  \]
  where $\partial\ell_{n}(\hat{\theta}^{ML})/\partial\theta=0$ is applied to eliminate the first-order term. By the positive definiteness of $\hat{H}_{n}\equiv H_{n}(\hat{\theta}^{ML})$, there exists a symmetric and positive definite matrix $\hat{G}_{n}\equiv G_{n}(\hat{\theta}^{ML})$ that satisfies $n\hat{H}_{n}=\hat{G}_{n}^{2}$. Define $\eta=\hat{G}_{n}(\theta-\hat{\theta}^{ML})$, and for simplicity, we write $\pi(\hat{\theta}^{ML})$ as $\hat{\pi}$. Then, we can evaluate the Bayes estimator as
  \begin{align}
    \notag
    &\hat{\theta}^{B}-\hat{\theta}^{ML}\\
    \label{eq:BayesEst.ML.Diff.Form1}
    &=\hat{G}_{n}^{-1}\frac{\int_{\mathbb{R}^{p}}\eta\left\{\hat{\pi}+\dot{\pi}(\hat{\theta}^{ML})^{\mathrm{\scriptscriptstyle{T}}}\hat{G}_{n}^{-1}\eta+O_{p}(n^{-1})\right\}\exp\left\{-2^{-1}\eta^{\mathrm{\scriptscriptstyle{T}}}\eta+g(\eta)+O_{p}(n^{-1})\right\}d\eta}{\int_{\mathbb{R}^{p}}\left\{\hat{\pi}+\dot{\pi}(\hat{\theta}^{ML})^{\mathrm{\scriptscriptstyle{T}}}\hat{G}_{n}^{-1}\eta+O_{p}(n^{-1})\right\}\exp\left\{-2^{-1}\eta^{\mathrm{\scriptscriptstyle{T}}}\eta+g(\eta)+O_{p}(n^{-1})\right\}d\eta}
    ,
  \end{align}
  where we defined $g(\eta)=6^{-1}(\sum_{r=1}^{p}\sum_{s=1}^{p}\hat{G}_{n}^{rs}\eta_{s}\partial/\partial\theta_{r})^{3}\ell_{n}(\theta)\mid_{\theta=\hat{\theta}^{ML}}$. Noting that expanding $\exp\{g(\eta)+O_{p}(n^{-1})\}$ yields $\exp\{g(\eta)+O_{p}(n^{-1})\}=1+g(\eta)+O_{p}(n^{-1})$, we have
  \begin{align*}
    \{\hat{\pi}+\dot{\pi}(\hat{\theta}^{ML})^{\mathrm{\scriptscriptstyle{T}}}\hat{G}_{n}^{-1}\eta+O_{p}(n^{-1})\}\exp\{g(\eta)+O_{p}(n^{-1})\}
    =\hat{\pi}+\dot{\pi}(\hat{\theta}^{ML})^{\mathrm{\scriptscriptstyle{T}}}\hat{G}_{n}^{-1}\eta+\hat{\pi}g(\eta)+O_{p}(n^{-1})
    .
  \end{align*}
  Thus, \eqref{eq:BayesEst.ML.Diff.Form1} can be written as
  \begin{align}
    \notag
    &\hat{\theta}^{B}-\hat{\theta}^{ML}\\
    \notag
    &=\hat{G}_{n}^{-1}\frac{\int_{\mathbb{R}^{p}}\eta\exp(-2^{-1}\eta^{\mathrm{\scriptscriptstyle{T}}}\eta)\{\dot{\pi}(\hat{\theta}^{ML})^{\mathrm{\scriptscriptstyle{T}}}\hat{G}_{n}^{-1}\eta+\hat{\pi}g(\eta)+O_{p}(n^{-1})\}d\eta}{(2\pi)^{p/2}\hat{\pi}+O_{p}(n^{-1})}\\
    \notag
    &=\frac{\hat{G}_{n}^{-1}}{\hat{\pi}}\int_{\mathbb{R}^{p}}\frac{\eta\exp(-2^{-1}\eta^{\mathrm{\scriptscriptstyle{T}}}\eta)}{(2\pi)^{p/2}}\left\{\dot{\pi}(\hat{\theta}^{ML})^{\mathrm{\scriptscriptstyle{T}}}\hat{G}_{n}^{-1}\eta+\hat{\pi}g(\eta)+O_{p}(n^{-1})\right\}d\eta+o_{p}(n^{-1})\\
    \notag
    &=\frac{\hat{G}_{n}^{-1}}{\hat{\pi}}\left\{\int_{\mathbb{R}^{p}}\frac{\eta\eta^{\mathrm{\scriptscriptstyle{T}}}\exp(-2^{-1}\eta^{\mathrm{\scriptscriptstyle{T}}}\eta)}{(2\pi)^{p/2}}d\eta\right\}\hat{G}_{n}^{-1}\dot{\pi}(\hat{\theta}^{ML})+\hat{G}_{n}^{-1}\int_{\mathbb{R}^{p}}\frac{\eta\exp(-2^{-1}\eta^{\mathrm{\scriptscriptstyle{T}}}\eta)}{(2\pi)^{p/2}}g(\eta)d\eta+o_{p}(n^{-1})\\
    \label{eq:BayesEst.ML.Diff.Form2}
    &=n^{-1}\hat{H}_{n}^{-1}\frac{\partial\log\pi(\hat{\theta}^{ML})}{\partial\theta}+\hat{G}_{n}^{-1}\int_{\mathbb{R}^{p}}\frac{\eta\exp(-2^{-1}\eta^{\mathrm{\scriptscriptstyle{T}}}\eta)}{(2\pi)^{p/2}}g(\eta)d\eta+o_{p}(n^{-1})
    .
  \end{align}
  The third equality follows from the symmetry of $\hat{G}_{n}^{-1}$, and the last equality follows from $\int_{\mathbb{R}^{p}}(2\pi)^{-p/2}\eta\eta^{\mathrm{\scriptscriptstyle{T}}}\exp\left(-2^{-1}\eta^{\mathrm{\scriptscriptstyle{T}}}\eta\right)d\eta$ being the identity matrix. The integral in the second term of \eqref{eq:BayesEst.ML.Diff.Form2} is calculated as $2^{-1}\hat{G}_{n}^{-1}\sum_{r=1}^{p}\sum_{s=1}^{p}\hat{H}_{n}^{rs}K_{n,\cdot rs}(\hat{\theta}^{ML})$ by the definition of $g(\eta)$. Thus, we obtain
  \begin{equation}\label{eq:BayesEst.ML.Diff.Form3}
    \hat{\theta}^{B}-\hat{\theta}^{ML}
    =n^{-1}\hat{H}_{n}^{-1}\left\{\frac{\partial\log\pi(\theta)}{\partial\theta}+\frac{1}{2}\sum_{r=1}^{p}\sum_{s=1}^{p}\hat{H}_{n}^{rs}K_{n,\cdot rs}(\hat{\theta}^{ML})\right\}+o_{p}(n^{-1})
    .
  \end{equation}
  By conditions (i) and (ii) of Assumption~\ref{asm:Limits1}, $\hat{H}_{n}$ can be approximated as $\hat{H}_{n}=H_{n}(\theta)+O_{p}(n^{-1/2})$ and consequently by condition (ii) of Assumption~\ref{asm:Limits1}, we have $\hat{H}_{n}^{-1}=H_{n}^{-1}(\theta)+O_{p}(n^{-1/2})$. Similarly, by conditions (i) and (ii) of Assumption~\ref{asm:Limits1}, we can approximate $K_{n,\cdot rs}(\hat{\theta}^{ML})$ as $K_{n,\cdot rs}(\hat{\theta}^{ML})=K_{n,\cdot rs}(\theta)+O_{p}(n^{-1/2})$. Substituting these approximations into \eqref{eq:BayesEst.ML.Diff.Form3}, we obtain \eqref{eq:BayesEst.ML.Diff}.
\end{proof}

Next, we evaluate $\hat{\theta}^{ML}-\theta$.

\begin{lemma}
  Under Assumption~\ref{asm:Limits1}, the difference between the maximum likelihood estimator and the true parameter can be approximated as
  \[
    \hat{\theta}^{ML}-\theta
    =n^{-1}H_{n}^{-1}(\theta)\left\{2\frac{\partial\ell_{n}(\theta)}{\partial\theta}+\frac{1}{2}\sum_{r=1}^{p}\sum_{s=1}^{p}H_{n}^{rs}(\theta)B_{n,rs}(\theta)\right\}+o_{p}(n^{-1})
    ,
  \]
  where $B_{n,rs}(\theta)$ is defined as
  \[
    B_{n,rs}(\theta)
    =2
    \begin{bmatrix}
      J_{n,1r,s}(\theta)\\
      \vdots\\
      J_{n,pr,s}(\theta)
    \end{bmatrix}
    +\sum_{t=1}^{p}\sum_{u=1}^{p}H_{n}^{tu}(\theta)I_{su}(\theta)K_{n,r\cdot t}(\theta)
    .
  \]
\end{lemma}
\begin{proof}
  By expanding $\partial\ell_{n}(\hat{\theta}^{ML})/\partial\theta=0$ around the true parameter, we obtain
  \begin{equation}\label{eq:ScoreEq.Expansion}
    0
    =\frac{\partial\ell_{n}(\theta)}{\partial\theta}-n\left\{H_{n}(\theta)-\frac{1}{2}\sum_{r=1}^{p}(\hat{\theta}_{r}^{ML}-\theta_{r})K_{n,r\cdot\cdot}(\theta)\right\}(\hat{\theta}^{ML}-\theta)+o_{p}(1)
    ,
  \end{equation}
  or equivalently,
  \begin{equation}\label{eq:ML.True.Diff.Form1}
    \hat{\theta}^{ML}-\theta
    =n^{-1}\left\{H_{n}(\theta)-\frac{1}{2}\sum_{r=1}^{p}(\hat{\theta}_{r}^{ML}-\theta_{r})K_{n,r\cdot\cdot}(\theta)\right\}^{-1}\left\{\frac{\partial\ell_{n}(\theta)}{\partial\theta}+o_{p}(1)\right\}
    .
  \end{equation}
  By conditions (i) and (ii) of Assumption~\ref{asm:Limits1}, we have
  \begin{align}
    \notag
    &H_{n}(\theta)\left\{H_{n}(\theta)-\frac{1}{2}\sum_{r=1}^{p}(\hat{\theta}_{r}^{ML}-\theta_{r})K_{n,r\cdot\cdot}(\theta)\right\}^{-1}\\
    \notag
    &=\left\{I_{p}-\frac{1}{2}\sum_{r=1}^{p}(\hat{\theta}_{r}^{ML}-\theta_{r})K_{n,r\cdot\cdot}(\theta)H_{n}^{-1}(\theta)\right\}^{-1}\\
    \notag
    &=I_{p}+\frac{1}{2}\sum_{r=1}^{p}(\hat{\theta}_{r}^{ML}-\theta_{r})K_{n,r\cdot\cdot}(\theta)H_{n}^{-1}(\theta)+O_{p}(n^{-1})\\
    \label{eq:ML.True.Diff.Int1}
    &=2I_{p}-H_{n}(\theta)H_{n}^{-1}(\theta)+\frac{1}{2}\sum_{r=1}^{p}(\hat{\theta}_{r}^{ML}-\theta_{r})K_{n,r\cdot\cdot}(\theta)H_{n}^{-1}(\theta)+O_{p}(n^{-1})
    .
  \end{align}
  Furthermore, conditions (i) and (ii) of Assumption~\ref{asm:Limits1} and \eqref{eq:ScoreEq.Expansion} imply
  \[
    0
    =\frac{\partial\ell_{n}(\theta)}{\partial\theta}-nH_{n}(\theta)(\hat{\theta}^{ML}-\theta)+O_{p}(1)
    .
  \]
  This can be rewritten as
  \[
    \hat{\theta}^{ML}-\theta
    =n^{-1}H_{n}^{-1}(\theta)\frac{\partial\ell_{n}(\theta)}{\partial\theta}+O_{p}(n^{-1})
    ,
  \]
  or equivalently, for each $r=1,\ldots,p$,
  \begin{equation}\label{eq:ML.True.Diff.Int2}
    \hat{\theta}_{r}^{ML}-\theta_{r}
    =n^{-1}\sum_{s=1}^{p}H_{n}^{rs}(\theta)\frac{\partial\ell_{n}(\theta)}{\partial\theta_{s}}+O_{p}(n^{-1})
    .
  \end{equation}
  Combining \eqref{eq:ML.True.Diff.Int2} with \eqref{eq:ML.True.Diff.Form1} and \eqref{eq:ML.True.Diff.Int1}, we have
  \[
    nH_{n}(\theta)(\hat{\theta}^{ML}-\theta)
    =2\frac{\partial\ell_{n}(\theta)}{\partial\theta}+\frac{1}{2}\sum_{r=1}^{p}\sum_{s=1}^{p}H_{n}^{rs}(\theta)B_{n,rs}(\theta)+o_{p}(1)
    .
  \]
  This completes the proof.
\end{proof}

\subsection{Proof of Corollary \texorpdfstring{\ref{cor:Bias.IndepOfN}}{2}}\label{subsubsec:Prf_cor:Bias.IndepOfN}

By applying Assumption~\ref{asm:Limits2} to Lemma~\ref{lem:BayesEst.ML.Diff}, we can straightforwardly evaluate the difference between the Bayes estimator and the maximum likelihood estimator as
\[
  \hat{\theta}^{B}-\hat{\theta}^{ML}
  =n^{-1}H^{-1}(\theta)\left\{\frac{\partial\log\pi(\theta)}{\partial\theta}+\frac{1}{2}\sum_{r=1}^{p}\sum_{s=1}^{p}H^{rs}(\theta)K_{\cdot rs}(\theta)\right\}+o_{p}(n^{-1})
  .
\]
It remains to evaluate the difference between the maximum likelihood estimator and the true parameter. First, we replace \eqref{eq:ML.True.Diff.Int1} with a slightly different expression, which is
\begin{align*}
  &H(\theta)\left\{H_{n}(\theta)-\frac{1}{2}\sum_{r=1}^{p}(\hat{\theta}_{r}^{ML}-\theta_{r})K_{n,r\cdot\cdot}(\theta)\right\}^{-1}\\
  &=\left[I_{p}-\{H(\theta)-H_{n}(\theta)\}H^{-1}(\theta)-\frac{1}{2}\sum_{r=1}^{p}(\hat{\theta}_{r}^{ML}-\theta_{r})K_{n,r\cdot\cdot}(\theta)H^{-1}(\theta)\right]^{-1}\\
  &=I_{p}+\{H(\theta)-H_{n}(\theta)\}H^{-1}(\theta)+\frac{1}{2}\sum_{r=1}^{p}(\hat{\theta}_{r}^{ML}-\theta_{r})K_{n,r\cdot\cdot}(\theta)H^{-1}(\theta)+O_{p}(n^{-1})\\
  &=2I_{p}-H_{n}(\theta)H^{-1}(\theta)+\frac{1}{2}\sum_{r=1}^{p}(\hat{\theta}_{r}^{ML}-\theta_{r})K_{n,r\cdot\cdot}(\theta)H^{-1}(\theta)+O_{p}(n^{-1})
  .
\end{align*}
This equation, combined with \eqref{eq:ML.True.Diff.Form1} and \eqref{eq:ML.True.Diff.Int2}, gives us
\begin{align*}
  &nH_{n}(\theta)(\hat{\theta}^{ML}-\theta)\\
  &=2\frac{\partial\ell_{n}(\theta)}{\partial\theta}+\sum_{r=1}^{p}\sum_{s=1}^{p}H^{rs}(\theta)
  \begin{bmatrix}
    J_{n,1r,s}(\theta)\\
    \vdots\\
    J_{n,pr,s}(\theta)
  \end{bmatrix}\\
  &\quad+\frac{1}{2}\sum_{r=1}^{p}\sum_{s=1}^{p}H_{n}^{rs}(\theta)\sum_{t=1}^{p}\sum_{u=1}^{p}H^{tu}K_{n,r\cdot t}(\theta)I_{n,su}(\theta)+o_{p}(1)\\
  &=2\frac{\partial\ell_{n}(\theta)}{\partial\theta}+\sum_{r=1}^{p}\sum_{s=1}^{p}H^{rs}(\theta)\left\{
  \begin{bmatrix}
    J_{n,1r,s}(\theta)\\
    \vdots\\
    J_{n,pr,s}(\theta)
  \end{bmatrix}
  +\frac{1}{2}\sum_{t=1}^{p}\sum_{u=1}^{p}H^{tu}(\theta)K_{r\cdot t}(\theta)I_{n,su}(\theta)
  \right\}+o_{p}(1)
  .
\end{align*}
Thus, we have
\[
  \hat{\theta}^{B}-\theta
  =n^{-1}H^{-1}(\theta)\left[\frac{\partial}{\partial\theta}\{\log\pi(\theta)+2\ell_{n}(\theta)\}+\frac{1}{2}\sum_{r=1}^{p}\sum_{s=1}^{p}H^{rs}(\theta)\bar{A}_{n,rs}(\theta)\right]+o_{p}(n^{-1})
  ,
\]
where
\[
  \bar{A}_{n,rs}(\theta)
  =
  \begin{bmatrix}
    K_{1rs}(\theta)+2J_{n,1r,s}(\theta)\\
    \vdots\\
    K_{prs}(\theta)+2J_{n,pr,s}(\theta)
  \end{bmatrix}
  +\frac{1}{2}\sum_{t=1}^{p}\sum_{u=1}^{p}H^{tu}(\theta)K_{r\cdot t}(\theta)I_{n,su}(\theta)
  .
\]
Taking the expectation of the above equation, we obtain the desired result.

\subsection{Proof of Theorem~\texorpdfstring{\ref{thm:ExistenceOfPrior}}{2}}\label{subsec:Proof.thm:ExistenceOfPrior}

We prove the following general result.

\begin{lemma}\label{lem:LinerPDE}
  Suppose $\Theta$ is a rectangular subspace of $\mathbb{R}^{p}$. Let $\phi:\Theta\to\mathbb{R}^{p}$ be a vector-valued function that is first-order differentiable. We refer to the $r$th component of $\phi(\theta)$ as $\phi_{r}(\theta)$. Suppose $\xi:\Theta\to\mathbb{R}$ is a real-valued function, and for all inner points $\theta\in\Theta$, consider a system of partial differential equations
  \begin{equation}\label{eq:LinearPDE}
    \frac{\partial\xi(\theta)}{\partial\theta}
    =\phi(\theta)
    .
  \end{equation}
  Suppose we can interchange the order of integration and differentiation of $\phi$, i.e.,
  \[
    \int\left(\frac{\partial\phi(\theta)}{\partial\theta}\right)d\theta
    =\frac{\partial}{\partial\theta}\left(\int\phi(\theta)d\theta\right)
    .
  \]
  Then, \eqref{eq:LinearPDE} has a twice continuously differentiable solution $\xi:\mathbb{\theta}\to\mathbb{R}$ if and only if
  \begin{equation}\label{eq:IntegrabilityCondition}
    \frac{\partial\phi_{r}(\theta)}{\partial\theta_{s}}
    =\frac{\partial\phi_{s}(\theta)}{\partial\theta_{r}}
    \quad(r,s=1,\ldots,p)
  \end{equation}
  holds.
\end{lemma}
\begin{proof}
  ($\Longrightarrow$): Differentiating both sides of \eqref{eq:LinearPDE} with respect to $\theta$, we have
  \[
    \frac{\partial\xi(\theta)}{\partial\theta\partial\theta^{\mathrm{\scriptscriptstyle{T}}}}
    =\frac{\partial\phi(\theta)}{\partial\theta}
    .
  \]
  Observe that the left-hand side of the above equation is a symmetric matrix since $\xi$ is twice continuously differentiable. Therefore, the condition \eqref{eq:IntegrabilityCondition} holds.

  ($\Longleftarrow$): Fix an arbitrary constant vector $(c_{1},\ldots,c_{p})\in\Theta$. For each $r=1,\ldots,p$, define a function $\psi_{r}:\mathbb{R}^{p-r+1}\to\mathbb{R}$ by $\psi_{r}(\theta_{r},\ldots,\theta_{p})=\phi_{r}(c_{1},\ldots,c_{r-1},\theta_{r},\ldots,\theta_{p})$. Then, for an arbitrary constant $C$, the function
  \begin{equation}\label{eq:LinearPDESolution}
    \xi(\theta)
    =\sum_{r=1}^{p}\int_{c_{r}}^{\theta_{r}}\psi_{r}(z,\theta_{r+1},\ldots,\theta_{p})dz+C
  \end{equation}
  is a solution to \eqref{eq:LinearPDE}. Indeed, by \eqref{eq:IntegrabilityCondition}, for $s>r$, we have
  \begin{equation}\label{eq:DerivativeOfPsi}
    \frac{\partial}{\partial \theta_{s}}\psi_{r}(\theta_{r},\ldots,\theta_{p})
    =\frac{\partial\phi_{r}(\theta)}{\partial \theta_{s}}\bigg|_{(\theta_{1},\ldots,\theta_{r-1})=(c_{1},\ldots,c_{r-1})}
    =\frac{\partial\phi_{s}(\theta)}{\partial \theta_{r}}\bigg|_{(\theta_{1},\ldots,\theta_{r-1})=(c_{1},\ldots,c_{r-1})}
    .
  \end{equation}
  Thus, for $s=1,\ldots,p$, we obtain
  \begin{align*}
    &\frac{\partial}{\partial\theta_{s}}\left\{\sum_{r=1}^{p}\int_{c_{r}}^{\theta_{r}}\psi_{r}(z,\theta_{r+1},\ldots,\theta_{p})dz+C\right\}\\
    &=\sum_{r=1}^{s-1}\int_{c_{r}}^{\theta_{r}}\left\{\frac{\partial}{\partial\theta_{s}}\psi_{r}(z,\theta_{r+1},\ldots,\theta_{p})\right\}dz+\psi_{s}(\theta_{s},\ldots,\theta_{p})\\
    &=\sum_{r=1}^{s-1}\int_{c_{r}}^{\theta_{r}}\left\{\frac{\partial}{\partial z}\phi_{s}(\theta_{1},\ldots,\theta_{r-1},z,\theta_{r+1},\ldots,\theta_{p})\bigg|_{(\theta_{1},\ldots,\theta_{r-1})=(c_{1},\ldots,c_{r-1})}\right\}dz+\psi_{s}(\theta_{s},\ldots,\theta_{p})\\
    &=\sum_{r=1}^{s-1}\{\phi_{s}(c_{1},\ldots,c_{r-1},\theta_{r},\ldots,\theta_{p})-\phi_{s}(c_{1},\ldots,c_{r},\theta_{r+1},\ldots,\theta_{p})\}
    +\phi_{s}(c_{1},\ldots,c_{s-1},\theta_{s},\ldots,\theta_{p})\\
    &=\phi_{s}(\theta_{1},\ldots,\theta_{p})
    ,
  \end{align*}
  where the second equality follows from \eqref{eq:DerivativeOfPsi} and the third equality follows from the fundamental theorem of calculus. Hence, \eqref{eq:LinearPDESolution} is a solution to \eqref{eq:LinearPDE}.
\end{proof}

\section{Further discussion on the main theoretical results}

In this section, we provide further discussion on the main theoretical results, as well as results concerning asymptotically unbiased priors for i.i.d. models. We begin by providing the following observation.

\begin{remark}\label{rem:H=I}
  It can be seen that when $H(\theta)=I(\theta)$ holds, we have $A_{rst}(\theta)=2\{K_{trs}(\theta)+J_{tr,s}(\theta)\}$. Hence, the definition of $\phi(\theta)$ in \eqref{eq:PriorCond.IndepOfN} simplifies to
  \begin{equation}\label{eq:phi_t.H=I}
    \phi_{t}(\theta)
    =-\sum_{r=1}^{p}\sum_{s=1}^{p}H^{rs}(\theta)\{K_{trs}(\theta)+J_{tr,s}(\theta)\}
    \quad(t=1,\ldots,p)
    .
  \end{equation}
\end{remark}

We move on to considering i.i.d. models. Suppose $X_{1},\ldots,X_{n}$ is a sequence of i.i.d. random variables with density function $f_{n}(x_{1},\ldots,x_{n}\mid\theta)=\prod_{i=1}^{n}f(x_{i}\mid\theta)$. Under some regularity conditions, it is well known that $\sqrt{n}(\hat{\theta}^{ML}-\theta)$ converges in distribution to $N(0,I^{-1}(\theta))$, where $I(\theta)$ is the Fisher information matrix defined by
\begin{equation}\label{eq:FisherInfo.iid}
  I(\theta)
  =E\left[\left\{\frac{\partial\log f(X\mid\theta)}{\partial\theta}\right\}\left\{\frac{\partial\log f(X\mid\theta)}{\partial\theta}\right\}^{\mathrm{\scriptscriptstyle{T}}}\right]
  =-E\left\{\frac{\partial^{2}\log f(X\mid\theta)}{\partial\theta\partial\theta^{\mathrm{\scriptscriptstyle{T}}}}\right\}
  .
\end{equation}
This means condition (i) of Assumption~\ref{asm:Limits1} is satisfied. We compute the limits defined in Assumption~\ref{asm:Limits2} in this situation. Since $X_{1},\ldots,X_{n}$ are i.i.d., applying the central limit theorem, we conclude that $\sqrt{n}\{H_{n}(\theta)-I(\theta)\}$ converges in distribution to a zero-mean normal distribution with some covariance function. Thus, condition (i) of Assumption~\ref{asm:Limits2} is satisfied with $H(\theta)=I(\theta)$. Applying the law of large numbers, we can show that condition (ii) of Assumption~\ref{asm:Limits2} is also satisfied with
\[
  K_{rst}(\theta)
  =E\left\{\frac{\partial^{3}\log f(X\mid\theta)}{\partial\theta_{r}\partial\theta_{s}\partial\theta_{t}}\right\}
  .
\]
Conditions (iii) and (iv) of Assumption~\ref{asm:Limits2} are straightforward; we can take $I(\theta)$ in condition (iii) as the Fisher information matrix in \eqref{eq:FisherInfo.iid}, and $J_{rs,t}(\theta)$ in condition (iv) as
\[
  J_{rs,t}(\theta)
  =E\left\{\frac{\partial^{2}\log f(X\mid\theta)}{\partial\theta_{r}\partial\theta_{s}}\frac{\partial\log f(X\mid\theta)}{\partial\theta_{t}}\right\}
  .
\]
Since $H(\theta)$ is identical to the Fisher information matrix $I(\theta)$, by Remark~\ref{rem:H=I}, $\phi(\theta)$ in \eqref{eq:PriorCond.IndepOfN} is simplified to \eqref{eq:phi_t.H=I}. Noting that the relation
\begin{equation}\label{eq:Relation.IJK}
  \frac{\partial I_{rs}(\theta)}{\partial\theta_{t}}+J_{rs,t}(\theta)+K_{rst}(\theta)=0
  \quad (r,s,t=1,\ldots,p)
\end{equation}
holds in general, we conclude the following:

\begin{corollary}\label{cor:Bias.iid}
  Given the assumptions previously discussed, the bias of the Bayes estimator is given by
  \[
    E(\hat{\theta}^{B})-\theta
    =n^{-1}I^{-1}(\theta)\left\{\frac{\partial\log\pi(\theta)}{\partial\theta}-\sum_{r=1}^{p}\sum_{s=1}^{p}I^{rs}(\theta)\frac{\partial}{\partial\theta_{s}}
    \begin{bmatrix}
      I_{1r}(\theta)\\
      \vdots\\
      I_{pr}(\theta)
    \end{bmatrix}
    \right\}+o(n^{-1})
    ,
  \]
  where $I(\theta)$ is the Fisher information matrix defined in \eqref{eq:FisherInfo.iid}.
  Consequently, the Bayes estimator is asymptotically unbiased if the prior $\pi(\theta)$ satisfies
  \begin{equation}\label{eq:PriorCond.iid}
    \frac{\partial\log\pi(\theta)}{\partial\theta_{t}}
    =\sum_{r=1}^{p}\sum_{s=1}^{p}I^{rs}(\theta)\frac{\partial I_{tr}(\theta)}{\partial\theta_{s}}
    \equiv\phi_{t}(\theta)
    \quad(t=1,\ldots,p)
    .
  \end{equation}
\end{corollary}

\begin{remark}
  The above result for i.i.d. models is consistent with the existing result; see Section 7 of \cite{hartigan1965asymptotically}.
\end{remark}

\begin{remark}\label{rem:Firth}
  There is an interesting connection between the asymptotically unbiased prior for i.i.d. models, Firth's method \citep{firth1993bias}, and the moment matching prior \citep{ghosh2011moment}. Firth's method reduces the bias of the maximum likelihood estimator to $o(n^{-1})$ by modifying the score equation $\partial\ell_{n}(\theta)/\partial\theta\mid_{\theta=\hat{\theta}^{ML}}=0$ as
  \begin{equation}\label{eq:FirthMethod}
    \frac{\partial\ell_{n}(\theta)}{\partial\theta_{t}}
    =-\sum_{r=1}^{p}\sum_{s=1}^{p}I^{rs}(\theta)\left\{J_{tr,s}(\theta)+\frac{1}{2}K_{trs}(\theta)\right\}
    \quad(t=1,\ldots,p)
    .
  \end{equation}
  In contrast, the moment matching prior $\pi^{M}(\theta)$ is defined such that the posterior mean matches the maximum likelihood estimator up to $o_{p}(n^{-1})$. This prior satisfies the equation
  \begin{equation}\label{eq:MomentMatchingPrior}
    \frac{\partial\log\pi^{M}(\theta)}{\partial\theta_{t}}
    =-\frac{1}{2}\sum_{r=1}^{p}\sum_{s=1}^{p}I^{rs}(\theta)K_{trs}(\theta)
    \quad(t=1,\ldots,p)
    .
  \end{equation}
  Interestingly, by considering the general relation \eqref{eq:Relation.IJK}, it can be found that the sum of the right-hand sides of \eqref{eq:FirthMethod} and \eqref{eq:MomentMatchingPrior} corresponds to the right-hand side of \eqref{eq:PriorCond.iid}. This observation suggests that the asymptotically unbiased prior can be interpreted as a bias-reduced version of the moment matching prior, where the bias reduction is achieved using Firth's method.
\end{remark}

\begin{remark}
  \cite{meng2002single} proposed the single observation unbiased prior, which is a prior that ensures unbiasedness of the Bayes estimator for a single observation. Our result is closely related to the concept of single observation unbiased prior in that the  single observation unbiased prior is always second-order unbiased when considering repeated sampling from the same distribution. Consequently, the necessary and sufficient condition for the existence of an asymptotically unbiased prior also serves as a necessary condition for the existence of a single observation unbiased prior. In fact, the asymptotically unbiased priors we derived often result in Bayes estimators that are exactly unbiased (see Section~\ref{sec:Examlpes.Appendix}).
\end{remark}

If the model is i.i.d. and the Fisher information matrix $I(\theta)$ is diagonal, a simpler sufficient condition for the integrability condition of Theorem~\ref{thm:ExistenceOfPrior} can be considered.

\begin{proposition}\label{prop:DiagonalFisherInfo}
  Suppose the model satisfies the assumptions of Theorem~\ref{thm:ExistenceOfPrior}. Suppose further that the model is i.i.d. and has a diagonal Fisher information matrix $I(\theta)$ that is twice continuously differentiable. Then, there exists an asymptotically unbiased prior if for any $t,u=1,\ldots,p$, the ratio $I_{tt}(\theta)/I_{uu}(\theta)$ can be expressed as the product of  the following two functions: (a) a twice continuously differentiable function that does not depend on $\theta_{t}$ (but may depend on $\{\theta_{r}:r\neq t\}$); (b) a twice continuously differentiable function that does not depend on $\theta_{u}$ (but may depend on $\{\theta_{r}:r\neq u\}$).
\end{proposition}
\begin{proof}
  Fix $t,u\in\{1,\ldots,p\}$. By assumption, the ratio $I_{tt}(\theta)/I_{uu}(\theta)$ can be expressed as
  \[
    I_{tt}(\theta)/I_{uu}(\theta)
    =k_{t}(\theta_{-t})k_{u}(\theta_{-u})
    ,
  \]
  where $k_{t}(\theta_{-t})$ and $k_{u}(\theta_{-u})$ satisfy conditions (a) and (b) of the statement of the proposition, respectively. Taking the logarithm of both sides, we have
  \[
    \log I_{tt}(\theta)-\log k_{t}(\theta_{-t})
    =\log I_{uu}(\theta)+\log k_{u}(\theta_{-u})
    .
  \]
  Furthermore, differentiating both sides with respect to $\theta_{t}$ and $\theta_{u}$, we obtain
  \[
    \frac{\partial^{2}}{\partial\theta_{u}\partial\theta_{t}}\{\log I_{tt}(\theta)-\log k_{t}(\theta_{-t})\}
    =\frac{\partial^{2}}{\partial\theta_{t}\partial\theta_{u}}\{\log I_{uu}(\theta)+\log k_{u}(\theta_{-u})\}
    ,
  \]
  that is,
  \[
    \frac{\partial}{\partial\theta_{u}}\left\{\frac{1}{I_{tt}(\theta)}\frac{\partial I_{tt}(\theta)}{\partial\theta_{t}}\right\}
    =\frac{\partial}{\partial\theta_{t}}\left\{\frac{1}{I_{uu}(\theta)}\frac{\partial I_{uu}(\theta)}{\partial\theta_{u}}\right\}
    .
  \]
  Since the Fisher information is diagonal, we have $I^{rr}(\theta)=1/I_{rr}(\theta)$ for each $r$, and $I_{rs}(\theta)=I^{rs}(\theta)=0$ for $r\neq s$. Thus, the integrability condition of Theorem~\ref{thm:ExistenceOfPrior} holds.
\end{proof}

\section{Derivations of asymptotically unbiased priors for various models}\label{sec:Examlpes.Appendix}

\subsection{One-parameter models}\label{subsec:OneParameterModels}

Suppose the model satisfies the assumptions of Theorem~\ref{thm:ExistenceOfPrior} with $p=1$. In this one-parameter case, since $\phi(\theta)$ is a scalar, the integrability condition is automatically satisfied, ensuring the existence of an asymptotically unbiased prior. Specifically, if the model is i.i.d., $\phi(\theta)$ defined in Corollary~\ref{cor:Bias.iid} is written as $\phi(\theta)=\phi_{1}(\theta)=I^{-1}(\theta)I'(\theta)$, where $I'(\theta)$ is the derivative of the Fisher information. An asymptotically unbiased prior can be constructed using the method in Corollary~\ref{cor:ConstructionOfPrior} as follows:
\begin{enumerate}
  \item Fix a constant $c\in\Theta$ arbitrarily.
  \item Define $\psi_{1}(\theta)=\phi_{1}(\theta)=I^{-1}(\theta)I'(\theta)$ and compute $\tilde{\pi}(\theta)=\exp\{\int_{c}^{\theta}\psi_{1}(z)dz\}=I(\theta)/I(c)$.
  \item Take the prior as $\pi(\theta)\propto I(\theta)$.
\end{enumerate}
Thus, when the model is i.i.d., the asymptotically unbiased prior is proportional to the Fisher information itself.

\begin{example}[Binomial distribution]
  Suppose $X_{1},\ldots,X_{n}$ are i.i.d. random variables following a Bernoulli distribution with parameter $\theta\in(0,1)$. By the previous argument, the asymptotically unbiased prior for $\theta$ is given by
  \[
    \pi(\theta)
    \propto I(\theta)
    =\theta^{-1}(1-\theta)^{-1}
    .
  \]
  Let $\bar{X}=n^{-1}\sum_{i=1}^{n}X_{i}$ denote the sample mean of the $n$ Bernoulli random variables. The posterior density is computed as
  \[
    \pi(\theta\mid X_{1},\ldots,X_{n})
    \propto\theta^{n\bar{X}-1}(1-\theta)^{n(1-\bar{X})-1}
    ,
  \]
  which is a beta distribution with parameters $(n\bar{X},n(1-\bar{X}))$. Thus, the posterior mean of $\theta$ is given by $\hat{\theta}^{B}=\bar{X}$, which is an unbiased estimator of $\theta$.
\end{example}

\subsection{Simple multi-parameter models}

In this section, we illustrate the construction of asymptotically unbiased priors for multi-parameter models. We show that, in certain cases, the priors constructed using the method described in Corollary~\ref{cor:ConstructionOfPrior} result in Bayes estimators that are exactly unbiased. Furthermore, we provide an example where an asymptotically unbiased prior does not exist.

\begin{example}[Mean and variance parameters of normal distribution]\label{ex:NormalMeanVar}
  Suppose $X_{1},\ldots,X_{n}$ are i.i.d. random variables with density function
  \[
    f(x\mid\theta)
    =\frac{1}{\sqrt{2\pi}\sigma}\exp\left\{-\frac{(x-\mu)^{2}}{2\sigma^{2}}\right\}
    ,
  \]
  where $\theta=(\mu,\sigma^{2})$ is the parameter of interest. We construct an asymptotically unbiased prior based on the result of Corollary~\ref{cor:ConstructionOfPrior}. The Fisher information $I(\theta)$ is given by
  \[
    I(\theta)
    =
    \begin{bmatrix}
      \sigma^{-2}&0\\
      0&\sigma^{-4}/2
    \end{bmatrix}
    =
    \begin{bmatrix}
      \theta_{2}^{-1}&0\\
      0&\theta_{2}^{-2}/2
    \end{bmatrix}
    .
  \]
  Observe that the Fisher information matrix is diagonal with $I_{11}(\theta)/I_{22}(\theta)=2\theta_{2}$. Thus, by Proposition~\ref{prop:DiagonalFisherInfo} and Theorem~\ref{thm:ExistenceOfPrior}, the existence of an asymptotically unbiased prior is guaranteed. We can also compute the inverse of the Fisher information matrix and its partial derivatives as
  \[
    I^{-1}(\theta)
    =
    \begin{bmatrix}
      \sigma^{2}&0\\
      0&2\sigma^{4}
    \end{bmatrix}
    ,\quad
    \frac{\partial I(\theta)}{\partial\theta_{1}}
    =
    \begin{bmatrix}
      0&0\\
      0&0
    \end{bmatrix}
    ,\quad
    \frac{\partial I(\theta)}{\partial\theta_{2}}
    =
    \begin{bmatrix}
      -\sigma^{-4}&0\\
      0&-\sigma^{-6}
    \end{bmatrix}
    ,
  \]
  respectively. According to Corollary~\ref{cor:ConstructionOfPrior}, we compute $\phi_{1}(\theta)$ and $\phi_{2}(\theta)$ as
  \[
    \phi_{1}(\theta)
    =\sum_{r=1}^{2}\sum_{s=1}^{2}I^{rs}(\theta)\frac{\partial I_{1r}(\theta)}{\partial\theta_{s}}
    =0
    ,\quad
    \phi_{2}(\theta)
    =\sum_{r=1}^{2}\sum_{s=1}^{2}I^{rs}(\theta)\frac{\partial I_{2r}(\theta)}{\partial\theta_{s}}
    =-2\sigma^{-2}
    .
  \]
  We follow the procedure of Corollary~\ref{cor:ConstructionOfPrior} to construct an asymptotically unbiased prior.
  \begin{enumerate}
    \item Fix an arbitrary constant $c>0$.
    \item Define $\psi_{1}(\theta_{1},\theta_{2})=0$ and $\psi_{2}(\theta_{2})=\psi_{2}(\sigma^{2})=-2\sigma^{-2}$. Compute
    \[
      \tilde{\pi}(\theta)
      \equiv\exp\left\{\int_{c}^{\sigma^{2}}\psi_{2}(z)dz\right\}=c^{2}/\sigma^{4}
      .
    \]
    \item Take $\pi(\theta)\propto\sigma^{-4}$.
  \end{enumerate}
  The obtained prior is different from Jeffreys' prior since Jeffreys' prior for this model is $\pi_{J}(\theta)\propto(\det I(\theta))^{1/2}\propto\sigma^{-3}$. Next, we calculate the posterior mean $\hat{\theta}^{B}$ corresponding to the prior
  \[
    \pi(\theta)
    =\pi(\mu,\sigma^{2})
    \propto\sigma^{-4}
    .
  \]
  Define $\bar{X}=n^{-1}\sum_{i=1}^{n}X_{i}$ and $S=\{\sum_{i=1}^{n}(X_{i}-\bar{X})^{2}\}^{1/2}$. Then, the posterior density is written as
  \[
    \pi(\mu,\sigma^{2}\mid X_{1},\ldots,X_{n})
    =\sqrt{\frac{n}{2\pi}}\frac{1}{\Gamma((n+1)/2)}\left(\frac{S^{2}}{2}\right)^{\frac{n+1}{2}}
    \sigma^{-n-4}\exp\left\{-\frac{S^{2}+n(\bar{X}-\mu)^{2}}{2\sigma^{2}}\right\}
    .
  \]
  Therefore, we obtain the posterior mean $\hat{\theta}^{B}=(\hat{\mu}^{B},\hat{\sigma}^{2,B})$ as
  \begin{align*}
    &\hat{\mu}^{B}
    =\int_{0}^{\infty}\int_{-\infty}^{\infty}\mu\pi(\mu,\sigma^{2}\mid X_{1},\ldots,X_{n})d\mu d\sigma^{2}
    =\bar{X}
    ,\\
    &\hat{\sigma}^{2,B}
    =\int_{0}^{\infty}\int_{-\infty}^{\infty}\sigma^{2}\pi(\mu,\sigma^{2}\mid X_{1},\ldots,X_{n})d\mu d\sigma^{2}
    =\frac{S^{2}}{n-1}
    .
  \end{align*}
  This estimator is an exactly unbiased estimator of $\theta=(\mu,\sigma^{2})$.
\end{example}

\begin{example}[Location and scale parameters of normal distribution]\label{ex:NormalLocScale}
  We consider the same model as in Example~\ref{ex:NormalMeanVar} but we are interested in the estimation of $\theta=(\mu,\sigma)$ instead of $(\mu,\sigma^{2})$. In this case, the Fisher information $I(\theta)$ is given by
  \[
    I(\theta)
    =
    \begin{bmatrix}
      \sigma^{-2}&0\\
      0&2\sigma^{-2}
    \end{bmatrix}
    =\theta_{2}^{-2}
    \begin{bmatrix}
      1&0\\
      0&2
    \end{bmatrix}
    .
  \]
  Since the Fisher information matrix is diagonal with $I_{11}(\theta)/I_{22}(\theta)=1/2$, the existence of an asymptotically unbiased prior is guaranteed by Proposition~\ref{prop:DiagonalFisherInfo} and Theorem~\ref{thm:ExistenceOfPrior}. Following Corollary~\ref{cor:ConstructionOfPrior} as in the previous example, we obtain the asymptotically unbiased prior
  \[
    \pi(\theta)
    =\pi(\mu,\sigma)
    \propto\sigma^{-2}
    .
  \]
  The prior above corresponds to Jeffreys' prior, unlike the parametrization of Example~\ref{ex:NormalMeanVar}.
  The posterior density is given by
  \[
    \pi(\mu,\sigma\mid X_{1},\ldots,X_{n})
    =\sqrt{\frac{2n}{\pi}}\frac{1}{\Gamma(n/2)}\left(\frac{S^{2}}{2}\right)^{\frac{n}{2}}
    \sigma^{-n-2}\exp\left\{-\frac{S^{2}+n(\bar{X}-\mu)^{2}}{2\sigma^{2}}\right\}
    ,
  \]
  and the posterior mean of $\theta=(\mu,\sigma)$ is computed as
  \begin{align*}
    &\hat{\mu}^{B}
    =\int_{0}^{\infty}\int_{-\infty}^{\infty}\mu\pi(\mu,\sigma^{2}\mid X_{1},\ldots,X_{n})d\mu d\sigma^{2}
    =\bar{X}
    ,\\
    &\hat{\sigma}^{B}
    =\int_{0}^{\infty}\int_{-\infty}^{\infty}\sigma^{2}\pi(\mu,\sigma^{2}\mid X_{1},\ldots,X_{n})d\mu d\sigma^{2}
    =\frac{\Gamma((n-1)/2)}{\sqrt{2}\Gamma(n/2)}S
    .
  \end{align*}
  Again, these are the exact unbiased estimators of $\mu$ and $\sigma$.
\end{example}

\begin{example}[Location and scale parameters of gamma distribution]
  Suppose $X_{1},\ldots,X_{n}$ are i.i.d. random variables with density function
  \[
    f(x\mid\theta)
    =\frac{x^{\theta_{1}-1}}{\Gamma(\theta_{1})\theta_{2}^{\theta_{1}}}\exp\left(-\frac{x}{\theta_{2}}\right)
    ,
  \]
  where $\theta=(\theta_{1},\theta_{2})$ is the parameter of interest. We look for a prior that satisfies the system of partial differential equations in Corollary~\ref{cor:ConstructionOfPrior}. Let $\Lambda$ denote the derivative of the digamma function, that is, $\Lambda(\theta_{1})=d^{2}\log\Gamma(\theta_{1})/d\theta_{1}^{2}$. Then, we can write the Fisher information $I(\theta)$ as
  \[
    I(\theta)
    =
    \begin{bmatrix}
      \Lambda(\theta_{1}) & \theta_{2}^{-1}\\
      \theta_{2}^{-1} & \theta_{1}\theta_{2}^{-2}
    \end{bmatrix}
    .
  \]
  The inverse of the Fisher information matrix and its partial derivatives can be computed as
  \begin{align*}
    &I^{-1}(\theta)
    =\frac{1}{\theta_{1}\Lambda(\theta_{1})-1}
    \begin{bmatrix}
      \theta_{1} & -\theta_{2}\\
      -\theta_{2} & \theta_{2}^{2}\Lambda(\theta_{1})
    \end{bmatrix}
    ,\\
    &\frac{\partial I(\theta)}{\partial\theta_{1}}
    =
    \begin{bmatrix}
      \frac{d}{d\theta_{1}}\Lambda(\theta_{1}) & 0\\
      0 & \theta_{2}^{-2}
    \end{bmatrix}
    ,\quad
    \frac{\partial I(\theta)}{\partial\theta_{2}}
    =
    \begin{bmatrix}
      0 & -\theta_{2}^{-2}\\
      -\theta_{2}^{-2} & -2\theta_{1}\theta_{2}^{-3}
    \end{bmatrix}
    ,
  \end{align*}
  respectively. Therefore, according to Corollary~\ref{cor:ConstructionOfPrior}, we compute $\phi(\theta)=(\phi_{1}(\theta),\phi_{2}(\theta))$ as
  \begin{align*}
    \phi_{1}(\theta)
    &=\sum_{r=1}^{2}\sum_{s=1}^{2}I^{rs}(\theta)\frac{\partial I_{1r}(\theta)}{\partial\theta_{s}}
    =\{\theta_{1}\Lambda(\theta_{1})-1\}^{-1}\left\{\theta_{1}\frac{d\Lambda(\theta_{1})}{d\theta_{1}}-\Lambda(\theta_{1})\right\},\\
    \phi_{2}(\theta)
    &=\sum_{r=1}^{2}\sum_{s=1}^{2}I^{rs}(\theta)\frac{\partial I_{2r}(\theta)}{\partial\theta_{s}}
    =\{\theta_{1}\Lambda(\theta_{1})-1\}^{-1}\{-2\theta_{1}\theta_{2}^{-1}\Lambda(\theta_{1})\}
    .
  \end{align*}
  While the derivative of $\phi_{1}$ with respect to $\theta_{2}$ is zero, the derivative of $\phi_{2}$ with respect to $\theta_{1}$ is not zero in general. This implies that the integrability condition is not satisfied. Hence, an asymptotically unbiased prior independent of $n$ does not exist for this model, according to Theorem~\ref{thm:ExistenceOfPrior}.
\end{example}

\subsection{Regression models}\label{subsec:RegModel}

\begin{example}[Linear regression model: Coefficients and error variance]\label{ex:LinearRegErrorVar}
  Consider the model
  \[
    y\sim N(X\beta,\sigma^{2}I_{n})
    ,
  \]
  where $y=(y_{1},\ldots,y_{n})$ is an $n$-dimensional random vector of response variables, $X=[x_{1}\ \cdots\ x_{n}]^{\mathrm{\scriptscriptstyle{T}}}$ is an $n\times p$ non-random matrix of explanatory variables and $\beta\in\mathbb{R}^{p}$ is the coefficient vector. We consider the estimation problem of the parameter $\theta=(\theta_{1},\ldots,\theta_{p},\theta_{p+1})=(\beta,\sigma^{2})\in\mathbb{R}^{p}\times(0,\infty)$. Since the model is not i.i.d., we cannot apply the result of \cite{hartigan1965asymptotically} given in Corollary~\ref{cor:Bias.iid}. Hence, we apply Corollary~\ref{cor:Bias.IndepOfN} to obtain an asymptotically unbiased prior.

  The log likelihood is given by
  \[
    \ell_{n}(\theta)
    =-\frac{n}{2}\log\sigma^{2}-\frac{1}{2\sigma^{2}}(y-X\beta)^{\mathrm{\scriptscriptstyle{T}}}(y-X\beta)+(const.)
    .
  \]
  We compute $H(\theta)$ and $I(\theta)$, which are found to be identical in this case, as
  \[
    H(\theta)
    =I(\theta)
    =
    \begin{bmatrix}
      \sigma^{-2}\lim_{n\to\infty}n^{-1}\sum_{i=1}^{n}x_{i}x_{i}^{\mathrm{\scriptscriptstyle{T}}}&0\\
      0&\sigma^{-4}/2
    \end{bmatrix}
    .
  \]
  To obtain $\phi(\theta)$ in Corollary~\ref{cor:Bias.IndepOfN}, we first compute $J_{tr,s}(\theta)$. We have
  \begin{align*}
    &J_{rs,t}(\theta)=0&(r,s=1,\ldots,p;\ t=1,\ldots,p+1),\\
    &J_{r(p+1),t}(\theta)=-\sigma^{-4}\lim_{n\to\infty}n^{-1}\sum_{i=1}^{n}x_{ir}x_{it}&(r,t=1,\ldots,p),\\
    &J_{r(p+1),(p+1)}(\theta)=0&(r=1,\ldots,p),\\
    &J_{(p+1)(p+1),t}(\theta)=3\sigma^{-4}\lim_{n\to\infty}n^{-1}\sum_{i=1}^{n}x_{it}&(t=1,\ldots,p),\\
    &J_{(p+1)(p+1),(p+1)}(\theta)=-\sigma^{-6}
    .\\
  \end{align*}
  Next, we compute $K_{rst}(\theta)$. We get
  \begin{align*}
    &K_{rst}(\theta)=0&(r,s,t=1,\ldots,p),\\
    &K_{rs(p+1)}(\theta)=\sigma^{-4}\lim_{n\to\infty}n^{-1}\sum_{i=1}^{n}x_{ir}x_{is}&(r,s=1,\ldots,p),\\
    &K_{r(p+1)(p+1)}(\theta)=0&(r=1,\ldots,p),\\
    &K_{(p+1)(p+1)(p+1)}(\theta)=2\sigma^{-6}
    .
  \end{align*}
  By substituting these values into the definition of $\phi_{t}$ given in Corollary~\ref{cor:Bias.IndepOfN}, we obtain
  \[
    \phi(\theta)
    =(0,\ldots,0,-2\sigma^{-2})
    .
  \]
  It is easy to verify that $\phi_{t}(\theta)$ is zero for all $t=1,\ldots,p$, and hence, the integrability condition is satisfied. Thus, we can construct an asymptotically unbiased prior using the construction method in Corollary~\ref{cor:ConstructionOfPrior} as
  \[
    \pi(\theta)
    =\pi(\beta,\sigma^{2})
    \propto\sigma^{-4}
    .
  \]
  The corresponding posterior distribution can be described as
  \[
    \beta\mid\sigma^{2},y
    \sim N\left(\hat{\beta}^{OLS},\sigma^{2}(X^{\mathrm{\scriptscriptstyle{T}}}X)^{-1}\right)
    ,\quad
    \sigma^{2}\mid y\sim IG\left(\frac{n-p}{2}+1,\frac{1}{2}y^{\mathrm{\scriptscriptstyle{T}}}(y-X\hat{\beta}^{OLS})\right)
    ,
  \]
  where $\hat{\beta}^{OLS}=(X^{\mathrm{\scriptscriptstyle{T}}}X)^{-1}X^{\mathrm{\scriptscriptstyle{T}}}y$ is the ordinary least squares estimator and $IG(a,b)$ is an inverse gamma distribution with shape parameter $a$ and scale parameter $b$. For $n>p$, the posterior mean under this prior is $(\hat{\beta}^{B},\hat{\sigma}^{2,B})=(\hat{\beta}^{OLS},y^{\mathrm{\scriptscriptstyle{T}}}(y-X\hat{\beta}^{OLS})/(n-p))$, which is an exactly unbiased estimator of $(\beta,\sigma^{2})$.
  The corresponding posterior distribution can be described as
  \[
    \beta\mid\sigma^{2},y
    \sim N\left(\hat{\beta}^{OLS},\sigma^{2}(X^{\mathrm{\scriptscriptstyle{T}}}X)^{-1}\right)
    ,\quad
    \sigma^{2}\mid y\sim IG\left(\frac{n-p}{2}+1,\frac{1}{2}y^{\mathrm{\scriptscriptstyle{T}}}(y-X\hat{\beta}^{OLS})\right)
    ,
  \]
  where $\hat{\beta}^{OLS}=(X^{\mathrm{\scriptscriptstyle{T}}}X)^{-1}X^{\mathrm{\scriptscriptstyle{T}}}y$ is the ordinary least squares estimator and $IG(a,b)$ is an inverse gamma distribution with shape parameter $a$ and scale parameter $b$. For $n>p$, the posterior mean under this prior is
  \[
    (\hat{\beta}^{B},\hat{\sigma}^{2,B})
    =\left(\hat{\beta}^{OLS},\frac{y^{\mathrm{\scriptscriptstyle{T}}}(y-X\hat{\beta}^{OLS})}{n-p}\right)
    ,
  \]
  which is an exactly unbiased estimator of $(\beta,\sigma^{2})$.
  The posterior mean is given by
  \[
    \hat{\theta}=(\hat{\beta}^{B},\hat{\sigma}^{B})=\left(\hat{\beta}^{OLS},\frac{\Gamma((n-p)/2)}{\sqrt{2}\Gamma((n-p+1)/2)}\{y^{\mathrm{\scriptscriptstyle{T}}}(y-X\hat{\beta}^{OLS})\}^{1/2}\right)
    ,
  \]
  which is an unbiased estimator of $(\beta,\sigma)$.
\end{example}

\begin{example}[Linear regression model: Coefficients and error standard deviation]
  Consider the same model as in Example~\ref{ex:LinearRegErrorVar}, but we are interested in the estimation of $\theta=(\beta,\sigma)\in\mathbb{R}^{p}\times(0,\infty)$ instead of $(\beta,\sigma^{2})$. Again, we apply Corollary~\ref{cor:Bias.IndepOfN} to obtain an asymptotically unbiased prior. By a similar computation as in Example~\ref{ex:LinearRegErrorVar}, we compute
  \[
    \phi(\theta)=(0,\ldots,0,-2\sigma^{-1})
    .
  \]
  Obviously, the integrability condition is satisfied in this case. We can construct an asymptotically unbiased prior using the construction method in Corollary~\ref{cor:ConstructionOfPrior} as
  \[
    \pi(\theta)
    =\pi(\beta,\sigma)
    \propto\sigma^{-2}
    .
  \]
  The corresponding posterior distribution can be described as
  \begin{align*}
    \beta\mid\sigma,y
    &\sim N\left(\hat{\beta}^{OLS},\sigma^{2}(X^{\mathrm{\scriptscriptstyle{T}}}X)^{-1}\right)
    ,\\
    \pi(\sigma\mid y)
    &=2\left\{\frac{y^{\mathrm{\scriptscriptstyle{T}}}(y-X \hat{\beta}^{OLS})}{2}\right\}^{\frac{n-p+1}{2}}\left\{\Gamma\left(\frac{n-p+1}{2}\right)\right\}^{-1}\\
    &\qquad\qquad\times\sigma^{-n+p-2}\exp\left\{-\frac{1}{2\sigma^{2}}y^{\mathrm{\scriptscriptstyle{T}}}(y-X \hat{\beta}^{OLS})\right\}
    ,
  \end{align*}
  and the posterior mean is given by
  \[
    \hat{\theta}=(\hat{\beta}^{B},\hat{\sigma}^{B})=\left(\hat{\beta}^{OLS},\frac{\Gamma((n-p)/2)}{\sqrt{2}\Gamma((n-p+1)/2)}\{y^{\mathrm{\scriptscriptstyle{T}}}(y-X\hat{\beta}^{OLS})\}^{1/2}\right)
    .
  \]
  This is an unbiased estimator of $(\beta,\sigma)$.
\end{example}

\begin{example}[Nested error regression model]
  We consider the nested error regression model
  \begin{align*}
    &y_{ij}=x_{ij}^{\top}\beta+v_{i}+\epsilon_{ij}
    \quad(i=1,\ldots,m;\ j=1\ldots,n_{i}),\\
    &v_{i}\sim N(0,\tau^{2})
    ,\quad
    \epsilon_{ij}\sim N(0,\sigma^{2})
    ,\quad
    v_{i}\perp\epsilon_{ij}
    ,
  \end{align*}
  where $y_{ij}$ is the response variable for each unit $j$ in area $i$, $x_{ij}$ is a $p$-dimensional non-random covariate vector, $\beta\in\mathbb{R}^{p}$ is the coefficient vector, $v_{i}$ represents the random effect for area $i$, and $\epsilon_{ij}$ is the error term. If we define
  \[
    y_{i}
    =(y_{i1},\ldots,y_{in_{i}})
    ,\quad
    x_{i}^{\top}
    =
    \begin{bmatrix}
      x_{i1}&\cdots&x_{in_{i}}
    \end{bmatrix}^{\top}
    ,\quad
    \epsilon_{i}=(\epsilon_{i1},\ldots,\epsilon_{in_{i}})
    ,
  \]
  we can express the model in vector form as
  \[
    y_{i}
    =x_{i}\beta+v_{i}\iota_{n_{i}}+\epsilon_{i}
    \quad(i=1,\ldots,m)
    ,
  \]
  where $\iota_{n_{i}}$ is an $n_{i}$-dimensional column vector of ones.
  The data are assumed to be independent across areas $i$. We consider the asymptotic setting where $m\to\infty$, with $n_{i}$ fixed. The parameter of interest is
  \[
    \theta
    =(\theta_{1},\ldots,\theta_{p},\theta_{p+1},\theta_
    {p+2})
    =(\beta,\tau^{2},\sigma^{2})
    \in\Theta=\mathbb{R}^{p}\times(0,\infty)^{2}
    .
  \]
  If we write
  \[
    V_{i}
    =V_{i}(\theta)
    =\tau^{2}\iota_{n_{i}}\iota_{n_{i}}^{\top}+\sigma^{2}I_{n_{i}}
    ,
  \]
  the log-likelihood function is given by
  \[
    \ell_{m}(\theta)
    =-\frac{1}{2}\sum_{i=1}^{m}\{\log\det(V_{i})+(y_{i}-x_{i}\beta)^{\mathrm{\scriptscriptstyle{T}}}V_{i}^{-1}(y_{i}-x_{i}\beta)\}+(const.)
    .
  \]

  We obtain an asymptotically unbiased prior using the result of Corollary~\ref{cor:Bias.IndepOfN}. The first step is to compute $\phi(\theta)$. The expression for $\phi(\theta)$ in Corollary~\ref{cor:Bias.IndepOfN} is given by
  \[
    \phi_{t}(\theta)
    =-\frac{1}{2}\sum_{r=1}^{p+2}\sum_{s=1}^{p+2}H^{rs}(\theta)A_{rst}(\theta)
    \quad(t=1,\ldots,p+2)
    ,
  \]
  where $A_{rs}(\theta)$ is defined as
  \[
    A_{rs}(\theta)
    =
    \begin{bmatrix}
      K_{1rs}(\theta)+2J_{1r,s}(\theta)\\
      \vdots\\
      K_{prs}(\theta)+2J_{pr,s}(\theta)\\
    \end{bmatrix}
    +\sum_{t=1}^{p+2}\sum_{u=1}^{p+2}H^{tu}(\theta)I_{su}(\theta)
    \begin{bmatrix}
      K_{rt1}(\theta)\\
      \vdots\\
      K_{rtp}(\theta)\\
    \end{bmatrix}
  \]
  and $H(\theta),I(\theta),J(\theta)$ and $K(\theta)$ are defined such that Assumption~\ref{asm:Limits2} is satisfied. By a simple calculation, we get
  \[
    H(\theta)
    =I(\theta)
    =\lim_{m\to\infty}\frac{1}{m}\sum_{i=1}^{m}
    \begin{bmatrix}
      x_{i}^{\mathrm{\scriptscriptstyle{T}}}V_{i}^{-1}x_{i}&0&0\\
      0&2^{-1}n_{i}^{2}(\sigma^{2}+n_{i}\tau)^{-2}&2^{-1}n_{i}(\sigma^{2}+n_{i}\tau)^{-2}\\
      0&2^{-1}n_{i}(\sigma^{2}+n_{i}\tau)^{-2}&2^{-1}\mathrm{Tr}(V_{i}^{-2})
    \end{bmatrix}
    .
  \]
  This means that the inverse of this matrix is represented as
  \[
    H^{-1}(\theta)
    =I^{-1}(\theta)
    =\begin{bmatrix}
      (\lim_{m\to\infty}m^{-1}\sum_{i=1}^{m}x_{i}^{\mathrm{\scriptscriptstyle{T}}}V_{i}^{-1}x_{i})^{-1}&0\\
      0&W(\theta)
    \end{bmatrix}
    ,
  \]
  where
  \begin{align*}
    W(\theta)
    &=w(\theta)\lim_{m\to\infty}\frac{1}{m}\sum_{i=1}^{m}
    \begin{bmatrix}
      \mathrm{Tr}(V_{i}^{-2})&-n_{i}(\sigma^{2}+n_{i}\tau)^{-2}\\
      -n_{i}(\sigma^{2}+n_{i}\tau)^{-2}&n_{i}^{2}(\sigma^{2}+n_{i}\tau)^{-2}
    \end{bmatrix}
    ,\\
    2w^{-1}(\theta)
    &=\left\{\lim_{m\to\infty}\frac{1}{m}\sum_{i=1}^{m}\frac{n_{i}^{2}}{(\sigma^{2}+n_{i}\tau^{2})^{2}}\right\}\left\{\lim_{m\to\infty}\frac{1}{m}\sum_{i=1}^{m}\mathrm{Tr}(V_{i}^{-2})\right\}\\
    &\qquad\qquad-\left\{\lim_{m\to\infty}\frac{1}{m}\sum_{i=1}^{m}\frac{n_{i}}{(\sigma^{2}+n_{i}\tau^{2})^{2}}\right\}^{2}
    .
  \end{align*}
  Since $H(\theta)$ and $I(\theta)$ are identical, by Remark~\ref{rem:H=I}, we have
  \[
    \phi_{t}(\theta)
    =-\sum_{r=1}^{p+2}\sum_{s=1}^{p+2}H^{rs}(\theta)\{K_{trs}(\theta)+J_{tr,s}(\theta)\}
    \quad(t=1,\ldots,p+2)
    .
  \]
  Furthermore, since $H^{-1}(\theta)$ is block diagonal, we have
  \begin{equation}\label{eq:NER.phi_t}
    \phi_{t}(\theta)
    =-\sum_{r=1}^{p}\sum_{s=1}^{p}H^{rs}(\theta)\{K_{trs}(\theta)+J_{tr,s}(\theta)\}-\sum_{r=p+1}^{p+2}\sum_{s=p+1}^{p+2}H^{rs}(\theta)\{K_{trs}(\theta)+J_{tr,s}(\theta)\}
  \end{equation}
  for $t=1,\ldots,p+2$. Thus, we only need to compute $K_{rst}(\theta),J_{tr,s}(\theta)$ for all $t=1,\ldots,p+2$ and $r,s=1,\ldots,p$, or $r,s=p+1,p+2$. This can be done by a simple calculation; for $r,s=1,\ldots,p$, we get
  \[
    J_{tr,s}(\theta)
    =
    \begin{cases}
      0&(t=1,\ldots,p),\\
      -\lim_{m\to\infty}\frac{1}{m}\sum_{i=1}^{m}(\sigma^{2}+n_{i}\tau^{2})^{-2}\iota_{n_{i}}^{\mathrm{\scriptscriptstyle{T}}}x_{i\cdot,r}x_{i\cdot,s}^{\mathrm{\scriptscriptstyle{T}}}\iota_{n_{i}}&(t=p+1),\\
      -\lim_{m\to\infty}\frac{1}{m}\sum_{i=1}^{m}x_{i\cdot,r}^{\mathrm{\scriptscriptstyle{T}}}V_{i}^{-2}x_{i\cdot,s}^{\mathrm{\scriptscriptstyle{T}}}&(t=p+2)
      ,
    \end{cases}
  \]
  and $K_{rst}(\theta)=-J_{tr,s}(\theta)$, where $x_{i\cdot,r}$ is the $r$th column vector of $x_{i}$. For $r,s=p+1,p+2$, we have
  \begin{align*}
    J_{t(p+1),p+1}(\theta)
    &=
    \begin{cases}
      0&(t=1,\ldots,p),\\
      -\lim_{m\to\infty}\frac{1}{m}\sum_{i=1}^{m}(\sigma^{2}+n_{i}\tau^{2})^{-3}n_{i}^{3}&(t=p+1),\\
      -\lim_{m\to\infty}\frac{1}{m}\sum_{i=1}^{m}(\sigma^{2}+n_{i}\tau^{2})^{-3}n_{i}^{2}&(t=p+2)
      ,
    \end{cases}
    \\
    K_{(p+1)(p+1)t}(\theta)
    &=-2J_{t(p+1),p+1}(\theta)
    ,\\
    J_{t(p+1),p+2}(\theta)
    &=
    \begin{cases}
      0&(t=1,\ldots,p),\\
      -\lim_{m\to\infty}\frac{1}{m}\sum_{i=1}^{m}(\sigma^{2}+n_{i}\tau^{2})^{-3}n_{i}^{2}&(t=p+1),\\
      -\lim_{m\to\infty}\frac{1}{m}\sum_{i=1}^{m}(\sigma^{2}+n_{i}\tau^{2})^{-3}n_{i}&(t=p+2)
      ,
    \end{cases}
    \\
    K_{(p+1)(p+2)t}(\theta)
    &=-2J_{t(p+1),p+2}(\theta)
    ,\\
    J_{t(p+2),p+1}(\theta)
    &=J_{t(p+1),p+2}(\theta)
    ,\\
    K_{(p+2)(p+1)t}(\theta)
    &=K_{(p+1)(p+2)t}(\theta)
    ,\\
    J_{t(p+2),p+2}(\theta)
    &=
    \begin{cases}
      0&(t=1,\ldots,p),\\
      -\lim_{m\to\infty}\frac{1}{m}\sum_{i=1}^{m}(\sigma^{2}+n_{i}\tau^{2})^{-3}n_{i}&(t=p+1),\\
      -\lim_{m\to\infty}\frac{1}{m}\sum_{i=1}^{m}\mathrm{Tr}(V_{i}^{-3})&(t=p+2)
      ,
    \end{cases}
    \\
    K_{(p+2)(p+2)t}(\theta)
    &=-2J_{t(p+2),p+2}(\theta)
    .
  \end{align*}
  Substituting these values into \eqref{eq:NER.phi_t} yields
  \[
    \phi_{t}(\theta)
    =0
    \quad(t=1,\ldots,p)
  \]
  and
  \begin{align*}
    -w^{-1}(\theta)\phi_{p+1}(\theta)
    &=\left\{\lim_{m\to\infty}\frac{1}{m}\sum_{i=1}^{m}\mathrm{Tr}(V_{i}^{-2})\right\}\left\{\lim_{m\to\infty}\frac{1}{m}\sum_{i=1}^{m}\frac{n_{i}^{3}}{(\sigma^{2}+n_{i}\tau^{2})^{3}}\right\}\\
    &\quad-2\left\{\lim_{m\to\infty}\frac{1}{m}\sum_{i=1}^{m}\frac{n_{i}^{2}}{(\sigma^{2}+n_{i}\tau^{2})^{3}}\right\}\left\{\lim_{m\to\infty}\frac{1}{m}\sum_{i=1}^{m}\frac{n_{i}}{(\sigma^{2}+n_{i}\tau^{2})^{2}}\right\}\\
    &\quad+\left\{\lim_{m\to\infty}\frac{1}{m}\sum_{i=1}^{m}\frac{n_{i}}{(\sigma^{2}+n_{i}\tau^{2})^{3}}\right\}\left\{\lim_{m\to\infty}\frac{1}{m}\sum_{i=1}^{m}\frac{n_{i}^{2}}{(\sigma^{2}+n_{i}\tau^{2})^{2}}\right\}
    ,
  \end{align*}
  \begin{align*}
    -w^{-1}(\theta)\phi_{p+2}(\theta)
    &=\left\{\lim_{m\to\infty}\frac{1}{m}\sum_{i=1}^{m}\mathrm{Tr}(V_{i}^{-2})\right\}\left\{\lim_{m\to\infty}\frac{1}{m}\sum_{i=1}^{m}\frac{n_{i}^{2}}{(\sigma^{2}+n_{i}\tau^{2})^{3}}\right\}\\
    &\quad-2\left\{\lim_{m\to\infty}\frac{1}{m}\sum_{i=1}^{m}\frac{n_{i}}{(\sigma^{2}+n_{i}\tau^{2})^{3}}\right\}\left\{\lim_{m\to\infty}\frac{1}{m}\sum_{i=1}^{m}\frac{n_{i}}{(\sigma^{2}+n_{i}\tau^{2})^{2}}\right\}\\
    &\quad+\left\{\lim_{m\to\infty}\frac{1}{m}\sum_{i=1}^{m}\mathrm{Tr}(V_{i}^{-3})\right\}\left\{\lim_{m\to\infty}\frac{1}{m}\sum_{i=1}^{m}\frac{n_{i}^{2}}{(\sigma^{2}+n_{i}\tau^{2})^{2}}\right\}
    .
  \end{align*}
  It can be verified that $\phi_{p+1}(\theta)$ and $\phi_{p+2}(\theta)$ are expressed as
  \[
    \phi_{p+1}(\theta)=\frac{\partial}{\partial\tau^{2}}\log g(\tau^{2},\sigma^{2})
    ,\quad
    \phi_{p+2}(\theta)=\frac{\partial}{\partial\sigma^{2}}\log\{\sigma^{-4}g(\tau^{2},\sigma^{2})\}
    ,
  \]
  where $g(\tau^{2},\sigma^{2})$ is given by
  \[
    g(\tau^{2},\sigma^{2})
    =g_{1}(\tau^{2},\sigma^{2})g_{2}(\tau^{2},\sigma^{2})\sigma^{4}+2g_{1}(\tau^{2},\sigma^{2})g_{3}(\tau^{2},\sigma^{2})\sigma^{2}\tau^{2}+g_{3}(\tau^{2},\sigma^{2})g_{4}(\tau^{2},\sigma^{2})\tau^{4}
    ,
  \]
  \begin{align*}
    &g_{1}(\tau^{2},\sigma^{2})
    =\lim_{m\to\infty}\frac{1}{m}\sum_{i=1}^{m}\frac{n_{i}(n_{i}-1)}{(\sigma^{2}+n_{i}\tau^{2})^{2}}
    ,\quad
    g_{2}(\tau^{2},\sigma^{2})
    =\lim_{m\to\infty}\frac{1}{m}\sum_{i=1}^{m}\frac{n_{i}}{(\sigma^{2}+n_{i}\tau^{2})^{2}}
    ,\\
    &g_{3}(\tau^{2},\sigma^{2})
    =\lim_{m\to\infty}\frac{1}{m}\sum_{i=1}^{m}\frac{n_{i}^{2}}{(\sigma^{2}+n_{i}\tau^{2})^{2}}
    ,\quad
    g_{4}(\tau^{2},\sigma^{2})
    =\lim_{m\to\infty}\frac{1}{m}\sum_{i=1}^{m}\frac{n_{i}^{2}(n_{i}-1)}{(\sigma^{2}+n_{i}\tau^{2})^{2}}
    .
  \end{align*}
  This representation of $\phi(\theta)$ implies that the integrability condition is satisfied, and thus an asymptotically unbiased prior exists.

  Finally, we compute the asymptotically unbiased prior according to Corollary~\ref{cor:ConstructionOfPrior}. For arbitrary constants $c_{1},c_{2}>0$, we define
  \[
    \psi_{p+1}(\theta_{p+1})=\frac{\partial}{\partial\tau^{2}}\log g(\tau^{2},\sigma^{2}),\quad
    \psi_{p+2}(\theta_{p+2})=\frac{\partial}{\partial\sigma^{2}}\log\{\sigma^{-4}g(c_{1},\sigma^{2})\}
    .
  \]
  Then, the asymptotically unbiased prior is
  \[
    \pi(\theta)
    \propto\exp\left(\int_{c_{1}}^{\tau^{2}}\psi_{p+1}(z)dz+\int_{c_{2}}^{\sigma^{2}}\psi_{p+2}(z)dz\right)
    \propto\sigma^{-4}g(\tau^{2},\sigma^{2})
    .
  \]
\end{example}

\begin{example}[Balanced nested error regression model]\label{ex:BalancedNER}
  We consider the case where $n_{i}\equiv n$ for all $i$ in the previous example. In this case, the calculation of the limit terms in $g(\tau^{2},\sigma^{2})$ is simplified, and $g(\tau^{2},\sigma^{2})$ reduces to
  \[
    g(\tau^{2},\sigma^{2})=n^{2}(n-1)(\sigma^{2}+n\tau^{2})^{-2}
    .
  \]
  Therefore, the asymptotically unbiased prior simplifies to
  \begin{equation}\label{eq:NER.Prior.CommonN}
    \pi(\theta)\propto\sigma^{-4}(\sigma^{2}+n\tau^{2})^{-2}
    .
  \end{equation}
\end{example}

\section{Discussion on the asymptotically unbiased prior for the balanced nested error regression model}\label{sec:DiscussNER}

\subsection{Posterior propriety}\label{app:NERPosteriorPropriety}

In this section, we prove the posterior propriety of the asymptotically unbiased prior for the balanced nested error regression model, which is given by \eqref{eq:NER.Prior.CommonN}. The corresponding posterior distribution of $\theta$ is given by
\begin{equation}\label{eq:NER.Posterior}
  \pi(\theta\mid\mathcal{D})
  \propto(\sigma^{2})^{-2-\frac{(n-1)m}{2}}(\sigma^{2}+n\tau^{2})^{-2-\frac{m}{2}}\exp\left\{-\frac{1}{2}\sum_{i=1}^{m}(y_{i}-x_{i}\beta)^{\mathrm{\scriptscriptstyle{T}}}V^{-1}(y_{i}-x_{i}\beta)\right\}
  ,
\end{equation}
where $V=\tau^{2}\iota_{n}\iota_{n}^{\mathrm{\scriptscriptstyle{T}}}+\sigma^{2}I_{n}$ is the covariance matrix of $y_{i}-x_{i}\beta$ with $\iota_{n}$ being an $n$-dimensional column vector of ones and $\mathcal{D}$ represents the data. We further consider a transformation of $\theta$ defined as $\bar{\theta}=(\beta,\rho,\sigma^{2})$ with $\rho\equiv\sigma^{2}/(\sigma^{2}+n\tau^{2})\in(0,1)$. The corresponding posterior distribution of $\bar{\theta}$ is
\begin{equation}\label{eq:NER.Posterior.ThetaBar}
  \bar{\pi}(\bar{\theta}\mid\mathcal{D})
  \propto(\sigma^{2})^{-3-\frac{nm}{2}}\rho^{\frac{m}{2}}\exp\left\{-\frac{1}{2}\sum_{i=1}^{m}(y_{i}-x_{i}\beta)^{\mathrm{\scriptscriptstyle{T}}}\bar{V}^{-1}(y_{i}-x_{i}\beta)\right\}
  \equiv\bar{\pi}(\bar{\theta}\mid\mathcal{D})
  ,
\end{equation}
where $\bar{V}^{-1}$ is defined as $\bar{V}^{-1}=\sigma^{-2}\{I_{n}-(1-\rho)\iota_{n}\iota_{n}^{\mathrm{\scriptscriptstyle{T}}}/n\}$.

To establish the posterior propriety, we need to show that the integral of $\bar{\pi}(\bar{\theta}\mid\mathcal{D})$ over $\bar{\Theta}$ is finite.
\begin{assumption}\label{asm:NERPosteriorPropriety}
  Assume that the following conditions are satisfied: (i) there exists $i\in\{1,\ldots,m\}$ such that $x_{i}$ has full column rank; (ii) there exists $i\in\{1,\ldots,m\}$ such that $[y_{i}\ x_{i}]$ has full column rank; (iii) there exists $i\in\{1,\ldots,m\}$ such that the sample variance covariance matrix $n^{-1}\sum_{j=1}^{n}(x_{ij}-n^{-1}\sum_{j=1}^{n}x_{ij})(x_{ij}-n^{-1}\sum_{j=1}^{n}x_{ij})^{\mathrm{\scriptscriptstyle{T}}}$ is positive semidefinite; (iv) there exists $i\in\{1,\ldots,m\}$ such that the sample variance covariance matrix
  \[
    n^{-1}\sum_{j=1}^{n}\left(
    \begin{bmatrix}y_{ij}\\x_{ij}\end{bmatrix}
    -n^{-1}\sum_{j=1}^{n}
    \begin{bmatrix}y_{ij}\\x_{ij}\end{bmatrix}
    \right)\left(
    \begin{bmatrix}y_{ij}\\x_{ij}\end{bmatrix}
    -n^{-1}\sum_{j=1}^{n}
    \begin{bmatrix}y_{ij}\\x_{ij}\end{bmatrix}
    \right)^{\mathrm{\scriptscriptstyle{T}}}
  \]
  is positive semidefinite.
\end{assumption}

\begin{proposition}
  If Assumption~\ref{asm:NERPosteriorPropriety} is satisfied, we have
  \[
    \int_{\bar{\Theta}}\bar{\pi}(\bar{\theta}\mid\mathcal{D})d\bar{\theta}<\infty
    .
  \]
\end{proposition}
\begin{proof}
  For notational simplicity, we define $y=[y_{1}^{\mathrm{\scriptscriptstyle{T}}}\ \cdots\ y_{m}^{\mathrm{\scriptscriptstyle{T}}}]^{\mathrm{\scriptscriptstyle{T}}}$ and $X=[x_{1}^{\mathrm{\scriptscriptstyle{T}}}\ \cdots\ x_{m}^{\mathrm{\scriptscriptstyle{T}}}]^{\mathrm{\scriptscriptstyle{T}}}$. Using this notation, the posterior distribution $\bar{\pi}(\bar{\theta}\mid\mathcal{D})$ can expressed as
  \[
    \bar{\pi}(\bar{\theta}\mid\mathcal{D})
    =(\sigma^{2})^{-3-\frac{nm}{2}}\rho^{\frac{m}{2}}\exp\left\{-\frac{1}{2\sigma^{2}}(y-X\beta)^{\mathrm{\scriptscriptstyle{T}}}\bar{Q}(\rho)(y-X\beta)\right\}
    ,
  \]
  where $\bar{Q}(\rho)$ is a block diagonal matrix with $m$ blocks, each given by $I_{n}-(1-\rho)\iota_{n}\iota_{n}^{\mathrm{\scriptscriptstyle{T}}}/n$. Since $X^{\mathrm{\scriptscriptstyle{T}}}\bar{Q}(\rho)X$ is positive definite for $\rho\in[0,1]$ by Lemmas~\ref{lem:Q(rho)>0} and \ref{lem:Q(0)>0}, we can integrate out $\beta$ as
  \begin{align*}
    &\int_{\mathbb{R}^{p}}\exp\left\{-\frac{1}{2\sigma^{2}}(y-X\beta)^{\mathrm{\scriptscriptstyle{T}}}\bar{Q}(\rho)(y-X\beta)\right\}d\beta\\
    &=(2\pi\sigma^{2})^{p/2}\det\{X^{\mathrm{\scriptscriptstyle{T}}}\bar{Q}(\rho)X\}^{-1/2}\exp\left(-\frac{1}{2\sigma^{2}}\left[y^{\mathrm{\scriptscriptstyle{T}}}\bar{Q}(\rho)y-y^{\mathrm{\scriptscriptstyle{T}}}\bar{Q}(\rho)X\{X^{\mathrm{\scriptscriptstyle{T}}}\bar{Q}(\rho)X\}^{-1}X^{\mathrm{\scriptscriptstyle{T}}}\bar{Q}(\rho)y\right]\right)\\
    &=(2\pi\sigma^{2})^{p/2}\det\{X^{\mathrm{\scriptscriptstyle{T}}}\bar{Q}(\rho)X\}^{-1/2}\exp\left\{-\frac{1}{2\sigma^{2}}h(\rho)\right\}
    ,
  \end{align*}
  where $h(\rho)$ is given by
  \[
    h(\rho)
    =\det\{X^{\mathrm{\scriptscriptstyle{T}}}\bar{Q}(\rho)X\}^{-1}\det\left\{\begin{bmatrix}y^{\mathrm{\scriptscriptstyle{T}}}\\X^{\mathrm{\scriptscriptstyle{T}}}\end{bmatrix}\bar{Q}(\rho)\begin{bmatrix}y&X\end{bmatrix}\right\}
    .
  \]
  Since Lemma~\ref{lem:Q(rho)>0} also implies
  \[
    \det\{X^{\mathrm{\scriptscriptstyle{T}}}\bar{Q}(\rho)X\}>0
    ,\quad
    \det\left\{\begin{bmatrix}y^{\mathrm{\scriptscriptstyle{T}}}\\X^{\mathrm{\scriptscriptstyle{T}}}\end{bmatrix}\bar{Q}(\rho)\begin{bmatrix}y&X\end{bmatrix}\right\}>0
    ,
  \]
  we can express the integral as
  \begin{align*}
    &\int_{\bar{\Theta}}\bar{\pi}(\bar{\theta}\mid\mathcal{D})d\bar{\theta}\\
    &=(2\pi)^{p/2}\int_{0}^{1}\int_{0}^{\infty}(\sigma^{2})^{-3-(nm-p)/2}\rho^{m/2}\det\{X^{\mathrm{\scriptscriptstyle{T}}}\bar{Q}(\rho)X\}^{-1/2}\exp\left\{-\frac{1}{2\sigma^{2}}h(\rho)\right\}d\sigma^{2}d\rho\\
    &=(2\pi)^{p/2}2^{2+(nm-p)/2}\Gamma\left(2+(nm-p)/2\right)\int_{0}^{1}\rho^{m/2}\det\{X^{\mathrm{\scriptscriptstyle{T}}}\bar{Q}(\rho)X\}^{-1/2}h(\rho)^{-\{2+(nm-p)/2\}}d\rho
    .
  \end{align*}
  Since the integrand is continuous and the domain of integration is bounded in the last expression, the integral is finite. Therefore, we can conclude that the posterior distribution is proper.
\end{proof}

\begin{lemma}\label{lem:Q(rho)>0}
  Suppose conditions (i) and (ii) of Assumption~\ref{asm:NERPosteriorPropriety} hold. Then, for any $\rho>0$, we have
  \[
    X^{\mathrm{\scriptscriptstyle{T}}}\bar{Q}(\rho)X>0
    ,\quad
    \begin{bmatrix}
      y^{\mathrm{\scriptscriptstyle{T}}}\\
      X^{\mathrm{\scriptscriptstyle{T}}}
    \end{bmatrix}
    \bar{Q}(\rho)
    \begin{bmatrix}
      y&X
    \end{bmatrix}
    >0
    .
  \]
\end{lemma}
\begin{proof}
  We first show the positive definiteness of $I_{n}-(1-\rho)\iota_{n}\iota_{n}^{\mathrm{\scriptscriptstyle{T}}}/n$. Indeed, for any non-zero vector $b\in\mathbb{R}^{n}\setminus\{0\}$, by the Cauchy--Schwarz inequality, we have
  \[
    b^{\mathrm{\scriptscriptstyle{T}}}\left(I_{n}-\frac{1-\rho}{n}\iota_{n}\iota_{n}^{\mathrm{\scriptscriptstyle{T}}}\right)b
    =\|b\|^{2}-\frac{1-\rho}{n}(\iota_{n}^{\mathrm{\scriptscriptstyle{T}}}b)^{2}
    \geq\|b\|^{2}-\frac{1-\rho}{n}\|\iota_{n}\|^{2}\|b\|^{2}
    =\rho\|b\|^{2}
    >0
    .
  \]
  Next, we note that the matrices $X^{\mathrm{\scriptscriptstyle{T}}}\bar{Q}(\rho)X$ and $[y\ X]^{\mathrm{\scriptscriptstyle{T}}}\bar{Q}(\rho)[y\ X]$ can be expressed as
  \begin{align*}
    X^{\mathrm{\scriptscriptstyle{T}}}\bar{Q}(\rho)X
    &=\sum_{i=1}^{m}x_{i}^{\mathrm{\scriptscriptstyle{T}}}\left(I_{n}-\frac{1-\rho}{n}\iota_{n}\iota_{n}^{\mathrm{\scriptscriptstyle{T}}}\right)x_{i}
    ,\\
    \begin{bmatrix}
      y^{\mathrm{\scriptscriptstyle{T}}}\\
      X^{\mathrm{\scriptscriptstyle{T}}}
    \end{bmatrix}
    \bar{Q}(\rho)
    \begin{bmatrix}
      y&X
    \end{bmatrix}
    &=\sum_{i=1}^{m}
    \begin{bmatrix}
      y_{i}^{\mathrm{\scriptscriptstyle{T}}}\\
      x_{i}^{\mathrm{\scriptscriptstyle{T}}}
    \end{bmatrix}
    \left(I_{n}-\frac{1-\rho}{n}\iota_{n}\iota_{n}^{\mathrm{\scriptscriptstyle{T}}}\right)
    \begin{bmatrix}
      y_{i}&x_{i}
    \end{bmatrix}
    .
  \end{align*}
  Observe that $x_{i}^{\mathrm{\scriptscriptstyle{T}}}\{I_{n}-(1-\rho)\iota_{n}\iota_{n}^{\mathrm{\scriptscriptstyle{T}}}/n\}x_{i}$ and $[y_{i}\ x_{i}]^{\mathrm{\scriptscriptstyle{T}}}\{I_{n}-(1-\rho)\iota_{n}\iota_{n}^{\mathrm{\scriptscriptstyle{T}}}/n\}[y_{i}\ x_{i}]$ are positive semidefinite for each $i$. Furthermore, by conditions (i) and (ii) of Assumption~\ref{asm:NERPosteriorPropriety}, there exists $i,i'$ such that
  \[
    x_{i}^{\mathrm{\scriptscriptstyle{T}}}\left(I_{n}-\frac{1-\rho}{n}\iota_{n}\iota_{n}^{\mathrm{\scriptscriptstyle{T}}}\right)x_{i}>0
    ,\quad
    \begin{bmatrix}
      y_{i'}^{\mathrm{\scriptscriptstyle{T}}}\\x_{i'}^{\mathrm{\scriptscriptstyle{T}}}
    \end{bmatrix}
    \left(I_{n}-\frac{1-\rho}{n}\iota_{n}\iota_{n}^{\mathrm{\scriptscriptstyle{T}}}\right)
    \begin{bmatrix}
      y_{i'}&x_{i'}
    \end{bmatrix}
    >0
  \]
  hold. Therefore, we conclude that both $X^{\mathrm{\scriptscriptstyle{T}}}\bar{Q}(\rho)X$ and $[y\ X]^{\mathrm{\scriptscriptstyle{T}}}\bar{Q}(\rho)[y\ X]$ are positive definite.
\end{proof}

\begin{lemma}\label{lem:Q(0)>0}
  Suppose Assumption~\ref{asm:NERPosteriorPropriety} holds. Then for $\rho=0$, we have
  \[
    X^{\mathrm{\scriptscriptstyle{T}}}\bar{Q}(0)X>0
    ,\quad
    \begin{bmatrix}
      y^{\mathrm{\scriptscriptstyle{T}}}\\
      X^{\mathrm{\scriptscriptstyle{T}}}
    \end{bmatrix}
    \bar{Q}(0)
    \begin{bmatrix}
      y&X
    \end{bmatrix}
    >0
    .
  \]
\end{lemma}
\begin{proof}
  By the same argument as in the previous lemma, we can show that $I_{n}-\iota_{n}\iota_{n}^{\mathrm{\scriptscriptstyle{T}}}/n$ is positive semidefinite. Thus, $x_{i}^{\mathrm{\scriptscriptstyle{T}}}(I_{n}-\iota_{n}\iota_{n}^{\mathrm{\scriptscriptstyle{T}}}/n)x_{i}\geq0$ and $[y_{i}\ x_{i}]^{\mathrm{\scriptscriptstyle{T}}}(I_{n}-\frac{1}{n}\iota_{n}\iota_{n}^{\mathrm{\scriptscriptstyle{T}}}/n)[y_{i}\ x_{i}]\geq0$ hold for each $i$. Furthermore, by condition (iii) of Assumption~\ref{asm:NERPosteriorPropriety}, there exists $i$ that satisfies
  \[
    x_{i}^{\mathrm{\scriptscriptstyle{T}}}\left(I_{n}-\frac{1}{n}\iota_{n}\iota_{n}^{\mathrm{\scriptscriptstyle{T}}}\right)x_{i}
    =\sum_{j=1}^{n}\left(x_{ij}-\frac{1}{n}\sum_{j=1}^{n}x_{ij}\right)\left(x_{ij}-\frac{1}{n}\sum_{j=1}^{n}x_{ij}\right)^{\mathrm{\scriptscriptstyle{T}}}
    >0
    .
  \]
  Therefore, we have
  \[
    X^{\mathrm{\scriptscriptstyle{T}}}\bar{Q}(0)X
    =\sum_{i=1}^{m}x_{i}^{\mathrm{\scriptscriptstyle{T}}}\left(I_{n}-\frac{1}{n}\iota_{n}\iota_{n}^{\mathrm{\scriptscriptstyle{T}}}\right)x_{i}
    >0
    .
  \]
  By a similar argument, we also obtain
  \[
    \begin{bmatrix}
      y^{\mathrm{\scriptscriptstyle{T}}}\\
      X^{\mathrm{\scriptscriptstyle{T}}}
    \end{bmatrix}
    \bar{Q}(0)
    \begin{bmatrix}
      y&X
    \end{bmatrix}
    =\sum_{i=1}^{m}
    \begin{bmatrix}
      y_{i}^{\mathrm{\scriptscriptstyle{T}}}\\
      x_{i}^{\mathrm{\scriptscriptstyle{T}}}
    \end{bmatrix}
    \left(I_{n}-\frac{1}{n}\iota_{n}\iota_{n}^{\mathrm{\scriptscriptstyle{T}}}\right)
    \begin{bmatrix}
      y_{i}&x_{i}
    \end{bmatrix}
    >0
  \]
  as required.
\end{proof}

\subsection{Sampling from the posterior distribution}\label{subsec:NERPosteriorSampling}

We demonstrate a Markov chain Monte Carlo algorithm to obtain samples from the posterior distribution in \eqref{eq:NER.Posterior}. Instead of directly sampling from $\pi(\theta\mid\mathcal{D})$, we consider a transformation of $\theta$ as in the previous section, and sample from the posterior distribution $\bar{\pi}(\bar{\theta}\mid\mathcal{D})$ in \eqref{eq:NER.Posterior.ThetaBar} using a Gibbs sampler. We can show that the full conditional distributions of $\beta,\rho,$ and $\sigma^{2}$ are given by
\begin{align*}
  &\beta\mid\rho,\sigma^{2},\mathcal{D}
  \sim N\left(\left(\sum_{i=1}^{m}x_{i}^{\mathrm{\scriptscriptstyle{T}}}\bar{V}^{-1}x_{i}\right)^{-1}\left(\sum_{i=1}^{m}x_{i}^{\mathrm{\scriptscriptstyle{T}}}\bar{V}^{-1}y_{i}\right),\ \left(\sum_{i=1}^{m}x_{i}^{\mathrm{\scriptscriptstyle{T}}}\bar{V}^{-1}x_{i}\right)^{-1}\right)
  ,\\
  &\rho\mid\beta,\sigma^{2},\mathcal{D}\sim TG_{(0,1)}\left(\frac{m}{2}+1,\ \frac{1}{2n\sigma^{2}}\sum_{i=1}^{m}(y_{i}-x_{i}\beta)^{\mathrm{\scriptscriptstyle{T}}}\iota_{n}\iota_{n}^{\mathrm{\scriptscriptstyle{T}}}(y_{i}-x_{i}\beta)\right)
  ,\\
  &\sigma^{2}\mid\beta,\rho,\mathcal{D}\sim IG\left(\frac{nm}{2}+2,\ \frac{1}{2}\sum_{i=1}^{m}(y_{i}-x_{i}\beta)^{\mathrm{\scriptscriptstyle{T}}}\left(I_{n}-\frac{1-\rho}{n}\iota_{n}\iota_{n}^{\mathrm{\scriptscriptstyle{T}}}\right)(y_{i}-x_{i}\beta)\right)
  ,
\end{align*}
where $TG_{(0,1)}(a,b)$ denotes a truncated gamma distribution with shape parameter $a$ and rate parameter $b$, truncated to the interval $(0,1)$, and $IG(a,b)$ denotes an inverse gamma distribution with shape parameter $a$ and scale parameter $b$.

\section{Details of the simulation studies}\label{app:Simulation}

\input{plots/AbsBias_beta}
\input{plots/MSE}
\input{plots/CoverageProbability}

In this section, we provide the details of the simulation studies in Section~\ref{sec:Simulation}.

We consider the balanced nested error regression model described in Example~\ref{ex:BalancedNER}. For simplicity, we assume that $\beta$ is two-dimensional. The true parameter values are set under the following two scenarios: (i) $\beta_{1}=\beta_{2}=\tau^{2}=\sigma^{2}=1$; (ii) $\beta_{1}=\beta_{2}=1, \tau^{2}=0.5, \sigma^{2}=4$. For each area $i$, the number of units is set to $n=5$, and the sample size is varied across $m\in\{10,32,100,316,1000\}$. The covariates are generated as $x_{ij}\sim N(\mu,\Sigma)$, where
\[
  \mu=(1,2),\quad
  \Sigma=\begin{bmatrix}4&1\\1&1\end{bmatrix}
  .
\]

We compare the performance of the following three priors for the parameter $\theta$: (1) the asymptotically unbiased prior: $\pi(\theta)\propto\{\sigma^{2}(\sigma^{2}+n\tau^{2})\}^{-2}$; (2) Jeffreys' prior for the variance components combined with a flat prior for the regression coefficients: $\pi(\theta)\propto\{\sigma^{2}(\sigma^{2}+n\tau^{2})\}^{-1}$. A similar composition of priors is employed in \cite{tiao1965bayesian}, where the random effects model of the form $y_{ij}=\mu+\alpha_{i}+\epsilon_{ij}$ is considered, and in their approach, a flat prior is used for $\mu$, while Jeffreys' prior is applied to the variance components; (3) the prior proposed by \cite{datta1991bayesian}: $\pi(\beta)\propto 1,\ \tau^{2}\sim IG(a_{\tau},b_{\tau})$, and $\sigma^{2}\sim IG(a_{\sigma},b_{\sigma})$. In our simulation studies, we set $a_{\tau}=b_{\tau}=a_{\sigma}=b_{\sigma}=5$.

For the asymptotically unbiased prior and Jeffreys' prior, the Markov chain Monte Carlo sample size is set to $N=2000$ with a warm-up size of $warmup=100$, while for the prior of \cite{datta1991bayesian}, $N=20000$ with $warmup=1000$. For the prior of \cite{datta1991bayesian}, a larger sample size is chosen due to the higher autocorrelation of the chain compared to the other priors. All generated chains have at least 300 effective sample sizes, indicating good convergence and reliable posterior inference.

To examine the frequentist properties of Bayes estimators, we simulate 10000 independent datasets of $\mathcal{D}=(y_{i},x_{i})_{i=1,\ldots,m}$ for each sample size $m$. For each dataset, we compute the posterior mean $\hat{\theta}^{B}$, the bias, the mean squared error, and the coverage probability of the 95\% credible interval for each parameter.

For each prior, Gibbs sampler is applied to sample from the posterior distribution. Sampling from the posterior distribution corresponding to the asymptotically unbiased prior is explained in Section~\ref{subsec:NERPosteriorSampling}, and a similar transformation of the parameter is used for Jeffreys' prior. For the prior of \cite{datta1991bayesian}, we treat the model as a hierarchical model and sample $\beta,\tau^{2},\sigma^{2},\{v_{i}:i=1,\ldots,m\}$ by Gibbs sampler.

Figure~\ref{fig:AbsBias_beta} shows the computed absolute bias of the posterior mean of $\beta_{1}$ and $\beta_{2}$ under the two parameter settings. It can be seen that the bias of $\beta$ remains small and stable across priors and sample sizes. Figure~\ref{fig:MSE} shows the log-mean squared error of the posterior mean and Fig.~\ref{fig:Coverage} shows the coverage probability of the 95\% credible interval. From these results, it can be observed that the asymptotically unbiased prior and Jeffreys' prior exhibit comparable performance in terms of mean squared error and coverage probability. However, a closer look reveals some minor differences. In terms of mean squared error (Fig.~\ref{fig:MSE}), the asymptotically unbiased prior shows a slightly better performance, particularly for the variance components. Conversely, regarding the coverage probability of the 95\% credible intervals (Fig.~\ref{fig:Coverage}), Jeffreys' prior tends to provide coverage closer to the nominal level. In contrast, the performance of the prior of \cite{datta1991bayesian} is more sensitive to the true parameter settings, which might be attributed to the choice of hyperparameters in the prior.

\bibliographystyle{chicago}
\bibliography{main}
\end{document}

%% file: plots/AbsBias.tex
\begin{figure}
\centering
\begin{subfigure}{0.24\textwidth}
\begin{tikzpicture}
\begin{axis}[
  title={(a) $\tau^{2}$},
  xlabel={\(\log_{10}(m)\)},
  xmode=log, log basis x=10,
  xtick={10,100,1000}, xticklabels={2,3,4},
  yticklabel style={/pgf/number format/fixed},
  width=4cm, height=4cm,
  mark size=1pt,
  font=\footnotesize,
  xlabel style={yshift=1ex},
  title style={yshift=-0.5ex}
]
\addplot[black, solid ,mark=*, mark size=1pt] coordinates {(10,0.003485) (32,0.001454) (100,0.000121) (316,0.000125) (1000,0.000267)};
\addplot[black, dashed,mark=square*, mark size=1pt] coordinates {(10,0.315819) (32,0.077155) (100,0.023542) (316,0.007469) (1000,0.002039)};
\addplot[black, dotted,mark=triangle*, mark size=1pt] coordinates {(10,0.128341) (32,0.053886) (100,0.018344) (316,0.005917) (1000,0.001560)};
\end{axis}
\end{tikzpicture}
\end{subfigure}
\begin{subfigure}{0.24\textwidth}
\begin{tikzpicture}
\begin{axis}[
  title={(b) $\sigma^{2}$},
  xlabel={\(\log_{10}(m)\)},
  xmode=log, log basis x=10,
  xtick={10,100,1000}, xticklabels={2,3,4},
  yticklabel style={/pgf/number format/fixed},
  width=4cm, height=4cm,
  mark size=1pt,
  font=\footnotesize,
  xlabel style={yshift=1ex},
  title style={yshift=-0.5ex}
]
\addplot[black, solid ,mark=*, mark size=1pt] coordinates {(10,0.002813) (32,0.000476) (100,0.000315) (316,0.000304) (1000,0.000048)};
\addplot[black, dashed,mark=square*, mark size=1pt] coordinates {(10,0.051696) (32,0.016431) (100,0.005349) (316,0.001281) (1000,0.000548)};
\addplot[black, dotted,mark=triangle*, mark size=1pt] coordinates {(10,0.041068) (32,0.015405) (100,0.005803) (316,0.001539) (1000,0.000648)};
\end{axis}
\end{tikzpicture}
\end{subfigure}
\begin{subfigure}{0.24\textwidth}
\begin{tikzpicture}
\begin{axis}[
  title={(c) $\tau^{2}$},
  xlabel={\(\log_{10}(m)\)},
  xmode=log, log basis x=10,
  xtick={10,100,1000}, xticklabels={2,3,4},
  yticklabel style={/pgf/number format/fixed},
  width=4cm, height=4cm,
  mark size=1pt,
  font=\footnotesize,
  xlabel style={yshift=1ex},
  title style={yshift=-0.5ex}
]
\addplot[black, solid ,mark=*, mark size=1pt] coordinates {(10,0.264572) (32,0.059093) (100,0.007575) (316,0.000974) (1000,0.000570)};
\addplot[black, dashed,mark=square*, mark size=1pt] coordinates {(10,0.583956) (32,0.129688) (100,0.029884) (316,0.008019) (1000,0.001637)};
\addplot[black, dotted,mark=triangle*, mark size=1pt] coordinates {(10,0.535357) (32,0.353673) (100,0.196900) (316,0.088244) (1000,0.032091)};
\end{axis}
\end{tikzpicture}
\end{subfigure}
\begin{subfigure}{0.24\textwidth}
\begin{tikzpicture}
\begin{axis}[
  title={(d) $\sigma^{2}$},
  xlabel={\(\log_{10}(m)\)},
  xmode=log, log basis x=10,
  xtick={10,100,1000}, xticklabels={2,3,4},
  yticklabel style={/pgf/number format/fixed},
  width=4cm, height=4cm,
  mark size=1pt,
  font=\footnotesize,
  xlabel style={yshift=1ex},
  title style={yshift=-0.5ex}
]
\addplot[black, solid ,mark=*, mark size=1pt] coordinates {(10,0.248593) (32,0.058750) (100,0.001317) (316,0.001271) (1000,0.000523)};
\addplot[black, dashed,mark=square*, mark size=1pt] coordinates {(10,0.006103) (32,0.013341) (100,0.019431) (316,0.005069) (1000,0.002524)};
\addplot[black, dotted,mark=triangle*, mark size=1pt] coordinates {(10,0.622056) (32,0.269144) (100,0.116380) (316,0.051173) (1000,0.017641)};
\end{axis}
\end{tikzpicture}
\end{subfigure}
\caption{Absolute bias of the posterior mean of $\tau^{2}$ and $\sigma^{2}$ when $n=5$: (1) the asymptotically unbiased prior (solid); (2) Jeffreys' prior for the variance components (dashed); (3) the prior of \cite{datta1991bayesian} (dotted). The true parameter values are (i) $\beta_{1}=\beta_{2}=\tau^{2}=\sigma^{2}=1$ for plots (a) and (b), and (ii) $\beta_{1}=\beta_{2}=1,\ \tau^{2}=0.5,\ \sigma^{2}=4$ for plots (c) and (d).}
  \label{fig:AbsBias}
\end{figure}

%% file: plots/AbsBias_beta.tex
\begin{figure}[tb]
\centering
\begin{subfigure}{0.24\textwidth}
\begin{tikzpicture}
\begin{axis}[
    title={(a) $\beta_{1}$},
    xlabel={\(\log_{10}(m)\)},
    xmode=log, log basis x=10,
    xtick={10,100,1000}, xticklabels={2,3,4},
    yticklabel style={/pgf/number format/fixed},
    width=4cm, height=4cm,
    mark size=1pt,
    font=\footnotesize,
    xlabel style={yshift=1ex},
    title style={yshift=-0.5ex}
]
\addplot[black, solid ,mark=*, mark size=1pt] coordinates {(10,0.000425) (32,0.000074) (100,0.000191) (316,0.000067) (1000,0.000022)};
\addplot[black, dashed,mark=square*, mark size=1pt] coordinates {(10,0.000426) (32,0.000077) (100,0.000190) (316,0.000067) (1000,0.000022)};
\addplot[black, dotted,mark=triangle*, mark size=1pt] coordinates {(10,0.000343) (32,0.000081) (100,0.000194) (316,0.000061) (1000,0.000027)};
\end{axis}
\end{tikzpicture}
\end{subfigure}
\begin{subfigure}{0.24\textwidth}
\begin{tikzpicture}
\begin{axis}[
    title={(b) $\beta_{2}$},
    xlabel={\(\log_{10}(m)\)},
    xmode=log, log basis x=10,
    xtick={10,100,1000}, xticklabels={2,3,4},
    yticklabel style={/pgf/number format/fixed},
    width=4cm, height=4cm,
    mark size=1pt,
    font=\footnotesize,
    xlabel style={yshift=1ex},
    title style={yshift=-0.5ex}
]
\addplot[black, solid ,mark=*, mark size=1pt] coordinates {(10,0.000055) (32,0.000029) (100,0.000217) (316,0.000364) (1000,0.000146)};
\addplot[black, dashed,mark=square*, mark size=1pt] coordinates {(10,0.000027) (32,0.000033) (100,0.000216) (316,0.000365) (1000,0.000146)};
\addplot[black, dotted,mark=triangle*, mark size=1pt] coordinates {(10,0.000088) (32,0.000025) (100,0.000220) (316,0.000356) (1000,0.000148)};
\end{axis}
\end{tikzpicture}
\end{subfigure}
\begin{subfigure}{0.24\textwidth}
\begin{tikzpicture}
\begin{axis}[
    title={(c) $\beta_{1}$},
    xlabel={\(\log_{10}(m)\)},
    xmode=log, log basis x=10,
    xtick={10,100,1000}, xticklabels={2,3,4},
    yticklabel style={/pgf/number format/fixed},
    width=4cm, height=4cm,
    mark size=1pt,
    font=\footnotesize,
    xlabel style={yshift=1ex},
    title style={yshift=-0.5ex}
]
\addplot[black, solid ,mark=*, mark size=1pt] coordinates {(10,0.001765) (32,0.000512) (100,0.000171) (316,0.000120) (1000,0.000083)};
\addplot[black, dashed,mark=square*, mark size=1pt] coordinates {(10,0.001764) (32,0.000493) (100,0.000171) (316,0.000121) (1000,0.000083)};
\addplot[black, dotted,mark=triangle*, mark size=1pt] coordinates {(10,0.001907) (32,0.000365) (100,0.000156) (316,0.000129) (1000,0.000083)};
\end{axis}
\end{tikzpicture}
\end{subfigure}
\begin{subfigure}{0.24\textwidth}
\begin{tikzpicture}
\begin{axis}[
    title={(d) $\beta_{2}$},
    xlabel={\(\log_{10}(m)\)},
    xmode=log, log basis x=10,
    xtick={10,100,1000}, xticklabels={2,3,4},
    yticklabel style={/pgf/number format/fixed},
    width=4cm, height=4cm,
    mark size=1pt,
    font=\footnotesize,
    xlabel style={yshift=1ex},
    title style={yshift=-0.5ex}
]
\addplot[black, solid ,mark=*, mark size=1pt] coordinates {(10,0.002829) (32,0.000008) (100,0.000029) (316,0.000412) (1000,0.000046)};
\addplot[black, dashed,mark=square*, mark size=1pt] coordinates {(10,0.002660) (32,0.000011) (100,0.000033) (316,0.000412) (1000,0.000046)};
\addplot[black, dotted,mark=triangle*, mark size=1pt] coordinates {(10,0.002289) (32,0.000017) (100,0.000056) (316,0.000421) (1000,0.000043)};
\end{axis}
\end{tikzpicture}
\end{subfigure}
\caption{Absolute bias of the posterior mean of $\beta_{1}$ and $\beta_{2}$ when $n=5$: (1) the asymptotically unbiased prior (solid); (2) Jeffreys' prior for the variance components (dashed); (3) the prior of \cite{datta1991bayesian} (dotted). The true parameter values are (i) $\beta_{1}=\beta_{2}=\tau^{2}=\sigma^{2}=1$ for plots (a) and (b), and (ii) $\beta_{1}=\beta_{2}=1,\ \tau^{2}=0.5,\ \sigma^{2}=4$ for plots (c) and (d).}
\label{fig:AbsBias_beta}
\end{figure}

%% file: plots/MSE.tex
\begin{figure}[tb]
\centering
\begin{subfigure}{0.24\textwidth}
\begin{tikzpicture}
\begin{axis}[
    title={(a) $\beta_{1}$},
    xlabel={\(\log_{10}(m)\)},
    xmode=log, log basis x=10,
    ymode=log, log basis y=10,
    ytick={1e-1,1e-2,1e-3,1e-4},
    yticklabels={-1,-2,-3,-4},
    xtick={10,100,1000}, xticklabels={1,2,3},
    yticklabel style={/pgf/number format/fixed},
    width=4cm, height=4cm,
    mark size=1pt,
    font=\footnotesize,
    xlabel style={yshift=1ex},
    title style={yshift=-0.5ex}
]
\addplot[black, solid ,mark=*, mark size=1pt] coordinates {(10,0.008386) (32,0.002461) (100,0.000782) (316,0.000248) (1000,0.000078)};
\addplot[black, dashed,mark=square*, mark size=1pt] coordinates {(10,0.008361) (32,0.002459) (100,0.000782) (316,0.000248) (1000,0.000078)};
\addplot[black, dotted,mark=triangle*, mark size=1pt] coordinates {(10,0.008264) (32,0.002452) (100,0.000781) (316,0.000248) (1000,0.000078)};
\end{axis}
\end{tikzpicture}
\end{subfigure}
\begin{subfigure}{0.24\textwidth}
\begin{tikzpicture}
\begin{axis}[
    title={(b) $\beta_{2}$},
    xlabel={\(\log_{10}(m)\)},
    xmode=log, log basis x=10,
    ymode=log, log basis y=10,
    ytick={1e-1,1e-2,1e-3,1e-4},
    yticklabels={-1,-2,-3,-4},
    xtick={10,100,1000}, xticklabels={1,2,3},
    yticklabel style={/pgf/number format/fixed},
    width=4cm, height=4cm,
    mark size=1pt,
    font=\footnotesize,
    xlabel style={yshift=1ex},
    title style={yshift=-0.5ex}
]
\addplot[black, solid ,mark=*, mark size=1pt] coordinates {(10,0.020387) (32,0.005884) (100,0.001895) (316,0.000568) (1000,0.000181)};
\addplot[black, dashed,mark=square*, mark size=1pt] coordinates {(10,0.020336) (32,0.0058808) (100,0.001894) (316,0.000568) (1000,0.000181)};
\addplot[black, dotted,mark=triangle*, mark size=1pt] coordinates {(10,0.019515) (32,0.005836) (100,0.001888) (316,0.000566) (1000,0.000180)};
\end{axis}
\end{tikzpicture}
\end{subfigure}
\begin{subfigure}{0.24\textwidth}
\begin{tikzpicture}
\begin{axis}[
    title={(c) $\tau^{2}$},
    xlabel={\(\log_{10}(m)\)},
    xmode=log, log basis x=10,
    ymode=log, log basis y=10,
    ytick={1,1e-1,1e-2,1e-3},
    yticklabels={0,-1,-2,-3},
    xtick={10,100,1000}, xticklabels={1,2,3},
    yticklabel style={/pgf/number format/fixed},
    width=4cm, height=4cm,
    mark size=1pt,
    font=\footnotesize,
    xlabel style={yshift=1ex},
    title style={yshift=-0.5ex}
]
\addplot[black, solid ,mark=*, mark size=1pt] coordinates {(10,0.305351) (32,0.095381) (100,0.029797) (316,0.008776) (1000,0.002933)};
\addplot[black, dashed,mark=square*, mark size=1pt] coordinates {(10,0.587382) (32,0.114594) (100,0.031580) (316,0.008944) (1000,0.002949)};
\addplot[black, dotted,mark=triangle*, mark size=1pt] coordinates {(10,0.076620) (32,0.051475) (100,0.023777) (316,0.008158) (1000,0.002863)};
\end{axis}
\end{tikzpicture}
\end{subfigure}
\begin{subfigure}{0.24\textwidth}
\begin{tikzpicture}
\begin{axis}[
    title={(d) $\sigma^{2}$},
    xlabel={\(\log_{10}(m)\)},
    xmode=log, log basis x=10,
    ymode=log, log basis y=10,
    ytick={1,1e-1,1e-2,1e-3},
    yticklabels={0,-1,-2,-3},
    xtick={10,100,1000}, xticklabels={1,2,3},
    yticklabel style={/pgf/number format/fixed},
    width=4cm, height=4cm,
    mark size=1pt,
    font=\footnotesize,
    xlabel style={yshift=1ex},
    title style={yshift=-0.5ex}
]
\addplot[black, solid ,mark=*, mark size=1pt] coordinates {(10,0.049920) (32,0.015506) (100,0.004932) (316,0.001585) (1000,0.000496)};
\addplot[black, dashed,mark=square*, mark size=1pt] coordinates {(10,0.05840) (32,0.016273) (100,0.005010) (316,0.001592) (1000,0.000497)};
\addplot[black, dotted,mark=triangle*, mark size=1pt] coordinates {(10,0.033089) (32,0.013415) (100,0.004719) (316,0.001561) (1000,0.000493)};
\end{axis}
\end{tikzpicture}
\end{subfigure}

\vspace{1em}

\begin{subfigure}{0.24\textwidth}
\begin{tikzpicture}
\begin{axis}[
    title={(e) $\beta_{1}$},
    xlabel={\(\log_{10}(m)\)},
    xmode=log, log basis x=10,
    ymode=log, log basis y=10,
    ytick={1e-1,1e-2,1e-3,1e-4},
    yticklabels={-1,-2,-3,-4},
    xtick={10,100,1000}, xticklabels={1,2,3},
    yticklabel style={/pgf/number format/fixed},
    width=4cm, height=4cm,
    mark size=1pt,
    font=\footnotesize,
    xlabel style={yshift=1ex},
    title style={yshift=-0.5ex}
]
\addplot[black, solid ,mark=*, mark size=1pt] coordinates {(10,0.028660) (32,0.008907) (100,0.002665) (316,0.000885) (1000,0.000275)};
\addplot[black, dashed,mark=square*, mark size=1pt] coordinates {(10,0.028801) (32,0.008907) (100,0.002665) (316,0.000885) (1000,0.000275)};
\addplot[black, dotted,mark=triangle*, mark size=1pt] coordinates {(10,0.028964) (32,0.008908) (100,0.002664) (316,0.000884) (1000,0.000275)};
\end{axis}
\end{tikzpicture}
\end{subfigure}
\begin{subfigure}{0.24\textwidth}
\begin{tikzpicture}
\begin{axis}[
    title={(f) $\beta_{2}$},
    xlabel={\(\log_{10}(m)\)},
    xmode=log, log basis x=10,
    ymode=log, log basis y=10,
    ytick={1e-1,1e-2,1e-3,1e-4},
    yticklabels={-1,-2,-3,-4},
    xtick={10,100,1000}, xticklabels={1,2,3},
    yticklabel style={/pgf/number format/fixed},
    width=4cm, height=4cm,
    mark size=1pt,
    font=\footnotesize,
    xlabel style={yshift=1ex},
    title style={yshift=-0.5ex}
]
\addplot[black, solid ,mark=*, mark size=1pt] coordinates {(10,0.037228) (32,0.011337) (100,0.003410) (316,0.001104) (1000,0.000353)};
\addplot[black, dashed,mark=square*, mark size=1pt] coordinates {(10,0.038018) (32,0.011360) (100,0.003411) (316,0.001104) (1000,0.000353)};
\addplot[black, dotted,mark=triangle*, mark size=1pt] coordinates {(10,0.037899) (32,0.011437) (100,0.003403) (316,0.001096) (1000,0.000350)};
\end{axis}
\end{tikzpicture}
\end{subfigure}
\begin{subfigure}{0.24\textwidth}
\begin{tikzpicture}
\begin{axis}[
    title={(g) $\tau^{2}$},
    xlabel={\(\log_{10}(m)\)},
    xmode=log, log basis x=10,
    ymode=log, log basis y=10,
    ytick={1,1e-1,1e-2,1e-3},
    yticklabels={0,-1,-2,-3},
    xtick={10,100,1000}, xticklabels={1,2,3},
    yticklabel style={/pgf/number format/fixed},
    width=4cm, height=4cm,
    mark size=1pt,
    font=\footnotesize,
    xlabel style={yshift=1ex},
    title style={yshift=-0.5ex}
]
\addplot[black, solid ,mark=*, mark size=1pt] coordinates {(10,0.286415) (32,0.085390) (100,0.034273) (316,0.011598) (1000,0.003624)};
\addplot[black, dashed,mark=square*, mark size=1pt] coordinates {(10,0.736013) (32,0.114530) (100,0.036771) (316,0.011801) (1000,0.003640)};
\addplot[black, dotted,mark=triangle*, mark size=1pt] coordinates {(10,0.313480) (32,0.145577) (100,0.050707) (316,0.013920) (1000,0.003762)};
\end{axis}
\end{tikzpicture}
\end{subfigure}
\begin{subfigure}{0.24\textwidth}
\begin{tikzpicture}
\begin{axis}[
    title={(h) $\sigma^{2}$},
    xlabel={\(\log_{10}(m)\)},
    xmode=log, log basis x=10,
    ymode=log, log basis y=10,
    ytick={1,1e-1,1e-2,1e-3},
    yticklabels={0,-1,-2,-3},
    xtick={10,100,1000}, xticklabels={1,2,3},
    yticklabel style={/pgf/number format/fixed},
    width=4cm, height=4cm,
    mark size=1pt,
    font=\footnotesize,
    xlabel style={yshift=1ex},
    title style={yshift=-0.5ex}
]
\addplot[black, solid ,mark=*, mark size=1pt] coordinates {(10,0.694959) (32,0.228050) (100,0.077464) (316,0.024998) (1000,0.008067)};
\addplot[black, dashed,mark=square*, mark size=1pt] coordinates {(10,0.736823) (32,0.235277) (100,0.078840) (316,0.025102) (1000,0.008082)};
\addplot[black, dotted,mark=triangle*, mark size=1pt] coordinates {(10,0.872265) (32,0.269151) (100,0.082189) (316,0.025570) (1000,0.008060)};
\end{axis}
\end{tikzpicture}
\end{subfigure}
\caption{Mean squared error of the posterior mean on a $\log_{10}$ scale when $n=5$: (1) the asymptotically unbiased prior (solid); (2) Jeffreys' prior for the variance components (dashed); (3) the prior of \cite{datta1991bayesian} (dotted). The true parameter values are (i) $\beta_{1}=\beta_{2}=\tau^{2}=\sigma^{2}=1$ for plots (a)--(d), and (ii) $\beta_{1}=\beta_{2}=1,\ \tau^{2}=0.5,\ \sigma^{2}=4$ for plots (e)--(f).}
\label{fig:MSE}
\end{figure}

%% file: plots/CoverageProbability.tex
\begin{figure}[tb]
\centering
\def\Nominal{\addplot[very thick,color=black,forget plot] coordinates {(10,0.95) (1000,0.95)};}
\begin{subfigure}{0.24\textwidth}
\begin{tikzpicture}
\begin{axis}[
    title={(a) $\beta_{1}$},
    xlabel={\(\log_{10}(m)\)},
    xmode=log,log basis x=10,
    xtick={10,100,1000},xticklabels={1,2,3},
    ymin=0.93,ymax=0.96,
    width=4cm,height=4cm,
    mark size=1pt,
    font=\footnotesize,
    xlabel style={yshift=1ex},
    title style={yshift=-0.5ex}
]
\Nominal
\addplot[black,  solid ,mark=*] coordinates {(10,0.9384) (32,0.9484) (100,0.9491) (316,0.9454) (1000,0.9475)};
\addplot[black,  dashed,mark=square*] coordinates {(10,0.9449) (32,0.9511) (100,0.9497) (316,0.9455) (1000,0.9476)};
\addplot[black,  dotted,mark=triangle*] coordinates {(10,0.9516) (32,0.9513) (100,0.9534) (316,0.9467) (1000,0.9418)};
\end{axis}
\end{tikzpicture}
\end{subfigure}
\begin{subfigure}{0.24\textwidth}
\begin{tikzpicture}
\begin{axis}[
    title={(b) $\beta_{2}$},
    xlabel={\(\log_{10}(m)\)},
    xmode=log,log basis x=10,
    xtick={10,100,1000},xticklabels={1,2,3},
    ymin=0.91,ymax=0.96,
    width=4cm,height=4cm,
    mark size=1pt,
    font=\footnotesize,
    xlabel style={yshift=1ex},
    title style={yshift=-0.5ex}
]
\Nominal
\addplot[black,  solid ,mark=*] coordinates {(10,0.9200) (32,0.9390) (100,0.9411) (316,0.9493) (1000,0.9460)};
\addplot[black,  dashed,mark=square*] coordinates {(10,0.9383) (32,0.9435) (100,0.9427) (316,0.9494) (1000,0.9462)};
\addplot[black,  dotted,mark=triangle*] coordinates {(10,0.9518) (32,0.9467) (100,0.9457) (316,0.9502) (1000,0.9480)};
\end{axis}
\end{tikzpicture}
\end{subfigure}
\begin{subfigure}{0.24\textwidth}
\begin{tikzpicture}
\begin{axis}[
    title={(c) $\tau^{2}$},
    xlabel={\(\log_{10}(m)\)},
    xmode=log,log basis x=10,
    xtick={10,100,1000},xticklabels={1,2,3},
    ymin=0.91,ymax=1.0,
    width=4cm,height=4cm,
    mark size=1pt,
    font=\footnotesize,
    xlabel style={yshift=1ex},
    title style={yshift=-0.5ex}
]
\Nominal
\addplot[black,  solid ,mark=*] coordinates {(10,0.9168) (32,0.9373) (100,0.9455) (316,0.9536) (1000,0.9446)};
\addplot[black,  dashed,mark=square*] coordinates {(10,0.9557) (32,0.9465) (100,0.9483) (316,0.9541) (1000,0.9456)};
\addplot[black,  dotted,mark=triangle*] coordinates {(10,0.9958) (32,0.9840) (100,0.9646) (316,0.9601) (1000,0.9482)};
\end{axis}
\end{tikzpicture}
\end{subfigure}
\begin{subfigure}{0.24\textwidth}
\begin{tikzpicture}
\begin{axis}[
    title={(d) $\sigma^{2}$},
    xlabel={\(\log_{10}(m)\)},
    xmode=log,log basis x=10,
    xtick={10,100,1000},xticklabels={1,2,3},
    ymin=0.94,ymax=0.99,
    width=4cm,height=4cm,
    mark size=1pt,
    font=\footnotesize,
    xlabel style={yshift=1ex},
    title style={yshift=-0.5ex}
]
\Nominal
\addplot[black,  solid ,mark=*] coordinates {(10,0.9426) (32,0.9506) (100,0.9507) (316,0.9485) (1000,0.9519)};
\addplot[black,  dashed,mark=square*] coordinates {(10,0.9514) (32,0.9535) (100,0.9515) (316,0.9501) (1000,0.9520)};
\addplot[black,  dotted,mark=triangle*] coordinates {(10,0.9810) (32,0.9654) (100,0.9562) (316,0.9516) (1000,0.9536)};
\end{axis}
\end{tikzpicture}
\end{subfigure}

\vspace{1em}
\begin{subfigure}{0.24\textwidth}
\begin{tikzpicture}
\begin{axis}[
    title={(e) $\beta_{1}$},
    xlabel={\(\log_{10}(m)\)},
    xmode=log,log basis x=10,
    xtick={10,100,1000},xticklabels={1,2,3},
    ymin=0.51,ymax=0.97,
    width=4cm,height=4cm,
    mark size=1pt,
    font=\footnotesize,
    xlabel style={yshift=1ex},
    title style={yshift=-0.5ex}
]
\Nominal
\addplot[black,  solid ,mark=*] coordinates {(10,0.9430) (32,0.9416) (100,0.9514) (316,0.9453) (1000,0.9501)};
\addplot[black,  dashed,mark=square*] coordinates {(10,0.9513) (32,0.9444) (100,0.9522) (316,0.9456) (1000,0.9501)};
\addplot[black,  dotted,mark=triangle*] coordinates {(10,0.5765) (32,0.5685) (100,0.5590) (316,0.5408) (1000,0.5257)};
\end{axis}
\end{tikzpicture}
\end{subfigure}
\begin{subfigure}{0.24\textwidth}
\begin{tikzpicture}
\begin{axis}[
    title={(f) $\beta_{2}$},
    xlabel={\(\log_{10}(m)\)},
    xmode=log,log basis x=10,
    xtick={10,100,1000},xticklabels={1,2,3},
    ymin=0.68,ymax=0.97,
    width=4cm,height=4cm,
    mark size=1pt,
    font=\footnotesize,
    xlabel style={yshift=1ex},
    title style={yshift=-0.5ex}
]
\Nominal
\addplot[black,  solid ,mark=*] coordinates {(10,0.9472) (32,0.9434) (100,0.9499) (316,0.9483) (1000,0.9482)};
\addplot[black,  dashed,mark=square*] coordinates {(10,0.9600) (32,0.9482) (100,0.9515) (316,0.9490) (1000,0.9485)};
\addplot[black,  dotted,mark=triangle*] coordinates {(10,0.7726) (32,0.7551) (100,0.7350) (316,0.7055) (1000,0.6863)};
\end{axis}
\end{tikzpicture}
\end{subfigure}
\begin{subfigure}{0.24\textwidth}
\begin{tikzpicture}
\begin{axis}[
    title={(g) $\tau^{2}$},
    xlabel={\(\log_{10}(m)\)},
    xmode=log,log basis x=10,
    xtick={10,100,1000},xticklabels={1,2,3},
    ymin=0.73,ymax=1.0,
    width=4cm,height=4cm,
    mark size=1pt,
    font=\footnotesize,
    xlabel style={yshift=1ex},
    title style={yshift=-0.5ex}
]
\Nominal
\addplot[black,  solid ,mark=*] coordinates {(10,0.9888) (32,0.9747) (100,0.9507) (316,0.9493) (1000,0.9512)};
\addplot[black,  dashed,mark=square*] coordinates {(10,0.9780) (32,0.9727) (100,0.9518) (316,0.9492) (1000,0.9525)};
\addplot[black,  dotted,mark=triangle*] coordinates {(10,0.7403) (32,0.8110) (100,0.8667) (316,0.9130) (1000,0.9397)};
\end{axis}
\end{tikzpicture}
\end{subfigure}
\begin{subfigure}{0.24\textwidth}
\begin{tikzpicture}
\begin{axis}[
    title={(h) $\sigma^{2}$},
    xlabel={\(\log_{10}(m)\)},
    xmode=log,log basis x=10,
    xtick={10,100,1000},xticklabels={1,2,3},
    ymin=0.82,ymax=0.96,
    width=4cm,height=4cm,
    mark size=1pt,
    font=\footnotesize,
    xlabel style={yshift=1ex},
    title style={yshift=-0.5ex}
]
\Nominal
\addplot[black,  solid ,mark=*] coordinates {(10,0.9180) (32,0.9442) (100,0.9515) (316,0.9506) (1000,0.9507)};
\addplot[black,  dashed,mark=square*] coordinates {(10,0.9478) (32,0.9498) (100,0.9513) (316,0.9508) (1000,0.9517)};
\addplot[black,  dotted,mark=triangle*] coordinates {(10,0.8222) (32,0.8952) (100,0.9267) (316,0.9375) (1000,0.9464)};
\end{axis}
\end{tikzpicture}
\end{subfigure}
\caption{Coverage probability of the 95\% credible interval for the posterior mean when $n=5$: (1) the asymptotically unbiased prior (solid); (2) Jeffreys' prior for the variance components (dashed); (3) the prior of \cite{datta1991bayesian} (dotted). The thick solid line represents the nominal level of 0.95. The true parameter values are (i) $\beta_{1}=\beta_{2}=\tau^{2}=\sigma^{2}=1$ for plots (a)--(d), and (ii) $\beta_{1}=\beta_{2}=1,\ \tau^{2}=0.5,\ \sigma^{2}=4$ for plots (e)--(f).}
\label{fig:Coverage}
\end{figure}